\documentclass[12pt,fleqn,leqno,letterpaper]{article}
\usepackage{amsmath,amssymb,amsthm}

\usepackage[dvips]{graphicx}
\usepackage[dvips]{color}
\usepackage{bbm}

\usepackage{soul}

\allowdisplaybreaks

\numberwithin{equation}{section}

\makeatletter
\renewcommand{\section}{\@startsection{section}{1}{0pt}{20pt}{6pt}{\large\bf}}
\renewcommand{\@seccntformat}[1]{\csname the#1\endcsname.\ }

\def\ni{\noindent}

\renewcommand{\epsilon}{\varepsilon}

\newcommand{\spint}[4]{\int_{#1}^{#2} \! #3 \, \mathrm{d} #4}

\newcommand{\rplusr}[1][]{\mathbb{R}_+ \times \mathbb{R}^{#1}}
\newcommand{\rplustor}{\mathbb{R}_+ \rightarrow \mathbb{R}}
\newcommand{\rtor}{\mathbb{R} \rightarrow \mathbb{R}}
\newcommand{\rplusrtor}[1][]{\mathbb{R}_+ \times \mathbb{R}^{#1} \rightarrow \mathbb{R}}

\newcommand{\indic}[1]{\, \mathbbm{1}_{\{ #1 \}}}

\newcommand{\lt}[2]{\ell^{#1}_{#2}}
\newcommand{\ltint}[6]{\int_{#1}^{#2} \!\! \spint{#3}{#4}{#5}{\lt{x}{s} #6}}
\newcommand{\sltint}[1]{\ltint{0}{T}{\mathbb{R}}{}{#1}{}}

\newtheorem{theorem}{Theorem}[section]

\newtheorem{lemma}[theorem]{Lemma}
\newtheorem{remark}[theorem]{Remark}
\newtheorem{definition}[theorem]{Definition}

\begin{document}

\title{\bf The Local Time-Space Integral and \\Stochastic Differential Equations}
\author{Daniel Wilson}
\date{}
\maketitle


%


{\par \leftskip=2.1cm \rightskip=2.1cm \footnotesize

\noindent Processes which arise as solutions to stochastic differential equations involving the local time (SDELTs), such as skew Brownian motion, are frequent sources of inspiration in theory and applications. Existence and uniqueness results for such equations rely heavily on the It\^o-Tanaka formula. Recent interest in time-inhomogeneous SDELTs indicates the need for comprehensive existence and uniqueness results in the time-dependent case, however, the absence of a suitable time-dependent It\^o-Tanaka formula forms a major barrier. Rigorously developing a two-parameter integral with respect to local time, known as the local time-space integral, we connect together and extend many known formulae from the literature and establish a general time-dependent It\^o-Tanaka formula. Then, we prove the existence of a unique strong solution for a large class of time-inhomogeneous SDELTs.
\par}


\section{Introduction}
In \cite{legall83}, Le Gall provided a general treatment of stochastic differential equations involving the local time (SDELTs), of the form
\begin{equation} \label{sdeltintro1}
X_t = X_0 + \spint{0}{t}{\sigma(X_s)}{B_s} + \spint{\mathbb{R}}{}{\lt{a}{t}(X)}{\nu(a)}.
\end{equation}
To leverage classical existence and uniqueness results, he applied a transformation to remove the local time term which uses the It\^o-Tanaka formula. Namely, for a continuous semimartingale $X$ and a function $F$ which is the difference of two convex functions, we have
\begin{equation} \label{ito-tanaka}
F(X_t) = F(X_0) + \spint{0}{t}{F'_-(X_s)}{X_s} + \frac12 \spint{\mathbb{R}}{}{\lt{a}{t}}{F'_-(a)},
\end{equation}
where $\mathrm{d}F'_-$ is the Lebesgue-Stieltjes measure associated to the left derivative of $F$. 

Inspired by the recent work of \'Etor\'e and Martinez \cite{etore18}, we consider SDELTs with time-dependent coefficients of the form
\begin{equation} \label{sdeltintro2}
X_t = X_0 + \spint{0}{t}{b(s,X_s)}{s} + \spint{0}{t}{\sigma(s,X_s)}{B_s} + \spint{\mathbb{R}}{}{\,\spint{0}{t}{h(s,a)}{_s \lt{a}{s}(X)}}{\nu(a)}.
\end{equation}
The corresponding transformation which removes the local time component is time-dependent, and hence the It\^o-Tanaka formula no longer applies. Many authors have continued the work of It\^o and Tanaka by developing correction terms involving the local time for time-dependent functions, but none are quite sufficient for this purpose. Notably however, Eisenbaum \cite{eisenbaum00} and Ghomrasni and Peskir \cite{ghomrasni03} derived a change of variables formula of the form
\begin{equation}  \label{ghompeskintro}
F(t,X_t) - F(0,X_0) = \spint{0}{t}{F_t(s,X_s)}{s} + \spint{0}{t}{F_x(s,X_s)}{X_s} - \frac12 \spint{0}{t}{\int_{\mathbb{R}} F_x(s,a)}{\lt{a}{s}},
\end{equation}
where the final term is an expression which aims to unify correction terms depending on the local time, the so-called `local time-space integral'. It was pointed out by Ghomrasni and Peskir that formal manipulations of this final term yield some interesting and useful formulae, including many time-dependent extensions of the It\^o-Tanaka formula from the literature. As a final remark, they state that `it is an interesting problem to establish these formulas rigorously under natural conditions'. As the local time is of unbounded variation in the $a$ variable, the Lebesgue-Stieltjes construction cannot be applied, and other methods are required. 

The aim of this paper is twofold. Firstly, we construct the local time-space integral by extension from simple functions in analogy with the construction of the stochastic integral for left-continuous integrands. This allows us to rigorously establish the formal manipulations of Peskir and Ghomrasni, and go beyond them to obtain new results. Secondly, we use these tools to establish existence and uniqueness results for a class of stochastic differential equations involving the local time (SDELTs), of the form (\ref{sdeltintro2}).

In Section 2, various results of other authors on change of variables including the local time are summarised. In Section 3, beginning from simple functions and employing discrete integration by parts, we demonstrate three representations for the final term in (\ref{ghompeskintro}). The representation
\begin{equation}
\spint{0}{t}{\int_{\mathbb{R}} F_x(s,a)}{\lt{a}{s}} = \spint{\mathbb{R}}{}{H(T,a)}{_a \lt{a}{T}} - \int_0^{T-} \left( \spint{\mathbb{R}}{}{ h(u,a) }{_a \lt{a}{u}} \right) \, \mathrm{d} \mu(u)
\end{equation} has appeared before in the special case when $\mu$ is Lebesgue measure, but the generality here is new. The representation 
\begin{equation} 
\spint{0}{t}{\int_{\mathbb{R}} F_x(s,a)}{\lt{a}{s}} = - \int_\mathbb{R} \spint{0}{T}{g(u,a)}{_u \lt{a}{u}} \, \mathrm{d} \nu(a). \end{equation}
is well known in the case of Lebesgue measure, and has been established by Eisenbaum \cite{eisenbaum00} in the case of Brownian motion, and by Peskir \cite{peskir05} for the case of a Dirac measure. Here we treat semimartingales and general Radon measures. 

In Section 4 we develop a change of variables formula, which takes the form (\ref{ghompeskintro}). In Section 5, we consider the local time-space integral as an operator, and look at properties of the map $(H,X) \mapsto \int_\mathbb{R} \spint{0}{T}{H(s,a)}{ \lt{a}{s}(X)}$. Importantly the final term in (\ref{ghompeskintro}) is of bounded variation in many cases, meaning it provides the Doob-Meyer semimartingale decomposition of $F(X)$ in analogy with the classical It\^o and It\^o-Tanaka formulae. Theorem \ref{ctsinintegrand}, giving sufficient conditions for a type of dominated convergence theorem, is new but specific to our approach. Theorem \ref{ctsinsemimg} is new to the author's knowledge, proving continuity as a function of the semimartingale. To provide a sufficient condition for this, we also derive a condition for convergence of the local time of a sequence of semimartingales in probability in Lemma \ref{convergenceoflt}, namely when
\begin{equation}
\lt{a}{t}(X^n) \xrightarrow{\enskip\mathbb{P}\enskip} \lt{a}{t}(X) \text{ as } n \to \infty.
\end{equation} 
Again, this is new to the author's knowledge. A natural condition on the sequence of semimartingales is uniformity in $\mathcal{H}^1$, which we introduce in Definition \ref{localh1}.

In Section 6, we demonstrate several results which show how the local time-space calculus can be applied. Theorem \ref{ltcformula} was proved and applied by Peskir in \cite{peskiramerican05} to prove the uniqueness of the optimal stopping boundary for the American put option with finite time horizon. Here we derive it quickly using the local time-space calculus. Theorem \ref{ghomrasnithm} was established by Ghomrasni \cite{ghomrasni10}, along with other results, and we repeat it in our setting with the same method of proof. Theorem \ref{prottersanmartin} was previously established directly by Protter and San Mart\'in \cite{protter93}, but we may derive the result using local time-space calculus and an application of Theorem \ref{ctsinsemimg}, showing the power of this approach.

Finally in Section 7, we provide a brief overview of SDELTs of the form (\ref{sdeltintro2}), before stating and proving results for weak existence and pathwise uniqueness for such equations. This provides the existence of a unique strong solution by the well-known method of Yamada-Watanabe. Lemma \ref{ltlemma} is of independent interest, giving an expression for the local time of a time-dependent function of a semimartingale, taking the form
\begin{equation}
\lt{F(t,a)}{t}(Y) = \spint{0}{t}{F_{x}(s,a)}{_s \lt{a}{s}(X)},
\end{equation}
where $Y = F(t,X_t)$.

\section{Local time and correction terms} \label{ltctsection}
This section gives an overview of local time integral terms which are relevant to the current work. We fix throughout a continuous semimartingale $X = (X_t)_{t \geq 0 }$ on a filtered probability space $(\Omega,\mathcal{F},(\mathcal{F}_t)_{t \geq  0},\mathbb{P})$. All functions are assumed to be Borel measurable. Following \cite{revuz04}, we define the right local time of $X$ as the almost-sure limit
\begin{equation} \label{localtimelimit} \lt{a}{t}(X) = \lim_{\epsilon \to 0} \, \frac{1}{\epsilon} \spint{0}{t}{\mathbbm{1}_{ \{ a \, \leq \, X_s  <  \,a + \epsilon \}  }}{\! \left< X,X \right>_s}, \end{equation}
for each $(t,a)   \in  \mathbb{R}_+ \! \times \mathbb{R}$. This coincides with its definition by means of the Tanaka formula 
\begin{equation}
L^a_t = |X_t - a| - |X_0 - a| - \int_0^t \, \text{sgn}(X_s - a) \, \mathrm{d} X_s,
\end{equation}
where we define the signum function
\begin{equation}
\text{sgn}(x) = \left\{ \begin{array}{cc}
1 & x > 0, \\
-1 & x \leq 0.
\end{array} \right.
\end{equation}
This field has a modification which is jointly continuous in $t$ and right continuous with left limits in $a$, which we always use. The time-dependent occupation time formula
\begin{equation} \label{otf}
\spint{0}{t}{G(s,X_s)}{\! \left< X,X \right>_s} = \, \spint{\mathbb{R}}{}{\spint{0}{t}{G(s,a)}{_s \lt{a}{s}}}{a},
\end{equation}
which holds for all bounded Borel measurable $G : \rplusrtor$, follows by the monotone class theorem after comparing It\^o's formula and the It\^o-Tanaka formula (\ref{ito-tanaka}). 

One may also use the right or symmetric local time by replacing the indicator in (\ref{localtimelimit}) by the right or symmetric form. This corresponds by replacing the signum function above by the obvious right continuous or symmetrically continuous alternative. The implications of this replacement are mentioned where relevant. For more details on the local time for continuous semimartingales, the reader may consult \cite{revuz04}.

A main strategy of proof involves the concept of weak convergence. This is more properly called weak* convergence, to use functional analytic terminology. Despite the fact that our measures are often signed measures rather than probability measures, we retain the probabilistic terminology and say that a sequence of measures $\mu^n$ converges to $\mu$ \textit{weakly} if $\mu^n(f) \to \mu(f)$ for all continuous functions $f$ of compact support.

Given a measure $\mu$, the measure $|\mu|$ denotes the total variation measure of $\mu$. If $\mu$ is defined on a measurable space $(S,\mathcal{S})$, then the total variation of $\mu$ is $|\mu|(S)$. Analogously, the measure $| \mathrm{d}f |$ denotes the total variation measure of $\mathrm{d} f$. The expression $\text{TV}( f )$ denotes the total variation of a function $f$ over its domain, which is equivalent to the total variation of $\mathrm{d} f$ as a measure.

\subsection{Bouleau-Yor}
 The Bouleau-Yor formula was first established in \cite{bouleau81}, but a transparent proof and further comments can be found in \cite{protter04}. Given a function $F \! : \! \rtor$, which is absolutely continuous with locally-bounded derivative $f$, we have that 
\begin{equation} \label{bouleau-yor}
F(X_t) - F(X_0) = \spint{0}{t}{f(X_s)}{X_s} - \frac12 \spint{\mathbb{R}}{}{f(a)}{_a \lt{a}{t}}.
\end{equation}
The final integral is defined as a limit in probability of approximating sums, which we will introduce in an analogous manner to construct the time-dependent version of this formula.

\subsection{Al-Hussaini and Elliot}
Using the Bouleau-Yor formula (\ref{bouleau-yor}) above, Al-Hussaini and Elliot \cite{hussaini87} proved that for a suitably regular function $F:\rplusrtor$, we have
\begin{equation} \label{al-h-elliot}
\begin{split}
F(t,X_t) - F(0,X_0) = &\spint{0}{t}{F_t(s,X_s)}{s} + \spint{0}{t}{F_x(s,X_s)}{X_s} \\ &- \frac12 \spint{\mathbb{R}}{}{F_x(t,a)}{_a \lt{a}{t}} + \frac12 \spint{0}{t}{ \int_{\mathbb{R}} F_{xt}(s,a)}{_a \lt{a}{s}} \, \mathrm{d}s,
\end{split}
\end{equation}
where the $\mathrm{d}_a \lt{a}{s}$ integral is the Bouleau-Yor integral. The final expression is a special case of the one defined in (\ref{densityexp3eqn}). Notably the work \cite{hussaini87} also contains some approximation results for the local time.

\subsection{Eisenbaum}
The formula established by Eisenbaum in \cite{eisenbaum00} is of the same form as (\ref{ghompesk}) below. The construction of the local time-space integral for Brownian motion follows from the representation 
\begin{equation}
\spint{0}{t}{\int_{\mathbb{R}} f(s,a)}{\lt{a}{s}} = \spint{0}{t}{f(s,B_s)}{B_s} - \spint{0}{t}{f(s,B_s)}{^* B_s},
\end{equation}
where the final integral is a backwards stochastic integral. A further interesting result is the representation of the local time-space integral for Brownian motion as an iterated integral with respect to a Radon measure in the same form as (\ref{iterated2}), which stems from the work of Azema, Jeulin, Knight and Yor \cite{azema98}. The idea of approximating by Riemann-type sums, in the same spirit as our construction, is also explored. 

In \cite{Eisenbaum06} and \cite{eisenbaum07}, the above representation is extended to  L\'evy processes (whose jumps are of bounded variation) and reversible semimartingales respectively. The construction of the local time-space integral makes use of the F\"ollmer-Protter-Shiryayev formula given in \cite{follmer95}, which depends on time-reversal of the process and hence cannot easily be generalised. The class of admissible functions $f$ is however particularly broad, and the corresponding change of variables formula requires only the existence of first-order Radon-Nikodym derivatives in each argument.

\subsection{Ghomrasni-Peskir}
The local time-space formula of Ghomrasni and Peskir \cite{ghomrasni03} is fundamentally the same as the one we obtain. If $F\! :  \rplusrtor$ is $C^1$ then
\begin{equation} \label{ghompesk}
F(t,X_t) - F(0,X_0) = \spint{0}{t}{F_t(s,X_s)}{s} + \spint{0}{t}{F_x(s,X_s)}{X_s} - \frac12 \spint{0}{t}{\int_{\mathbb{R}} F_x(s,a)}{\lt{a}{s}}.
\end{equation}
The construction of the final integral makes use of the local time on curves formula, which instead appears in our case as a consequence of the local time-space calculus, Theorem \ref{ltcformula}. The latter part of \cite{ghomrasni03} is devoted to demonstrating, mainly through non-rigorous manipulations, many of the other formulae that appear in this work.

\subsection{Elworthy-Truman-Feng-Zhao}
In \cite{elworthy05}, Elworthy, Truman and Zhao construct an integral of the form (\ref{densityexp3eqn}) below, and prove a corresponding change of variables formula. They then generalise this formula to higher dimensions in \cite{feng07}, using single parameter integrals which contain the components of the process along with their one-dimensional local times. Further, via a generalisation of the multi-parameter Young integral, Feng and Zhao directly construct the integral
\begin{equation}
\int \spint{[0,T] \times \mathbb{R}}{}{H(s,a)}{_{(s,a)}\lt{a}{s}}.
\end{equation}
We refer the reader to \cite{feng06} for this approach, which is not taken up in this work. In \cite{feng08} and \cite{feng10}, the local-time integral in (\ref{bouleau-yor}) is expressed as a rough-path integral, for continuous semimartingales and a certain class of L\'evy processes respectively.

\section{Constructing the local time-space integral}
Below we will define the local time-space integral. This is motivated by the work of Ghomrasni and Peskir \cite{ghomrasni03}, where it was noticed that an integration-by-parts procedure applied to the formal expression
\begin{equation}
\sltint{F_x(s,x)}
\end{equation}
yields many of the previous expressions given in Section \ref{ltctsection}. We take this as inspiration for the construction, starting from the definition for simple sums and then employing discrete integration by parts and approximation by Riemann-type sums. It is remarkable that the naive equation (\ref{simpledef}), introduced by Eisenbaum \cite{eisenbaum00}, turns out to be the correct definition to unify the other It\^o formulae with local time correction terms.

Let us briefly note that the Bouleau-Yor integral is well defined by means of (\ref{bouleau-yor}), and versions of Fubini's theorem and the dominated convergence theorem for the $d_a \lt{a}{s}$ integral follow from the stochastic and deterministic versions.
\begin{definition} \label{bouleau-yor-def}
Given a locally-bounded function $f: \mathbb{R} \to \mathbb{R}$, we define the Bouleau-Yor integral of $f$ by
\begin{equation} \label{bouleau-yor2}
\spint{\mathbb{R}}{}{f(a)}{_a \lt{a}{t}} = 2 \left[ F(X_t) - F(X_0) - \spint{0}{t}{f(X_s)}{X_s} \right],
\end{equation}
where $F$ is any antiderivative of $f$.
\end{definition}

The fundamental definition of the local time-space integral of a product of indicator functions as introduced by Eisenbaum \cite{eisenbaum00} is
\begin{equation} \label{simpledef} \Lambda(\mathbbm{1}_{ (s,t] } \, \mathbbm{1}_{(x,y]}) = \sltint{\mathbbm{1}_{ (s,t] } \, \mathbbm{1}_{(x,y]}} =\lt{y}{t} - \lt{y}{s} - \lt{x}{t} + \lt{x}{s}. \end{equation}
This is extended to simple functions by linearity. We will use either of these two expressions to denote the local time-space integral. By considering the right-hand side of (\ref{simpledef}) above as a Lebesgue-Stieltjes integral, or Bouleau-Yor integral given by (\ref{bouleau-yor2}), we are led to the following three definitions. If a function $H$ satisfies the conditions of any of these definitions, we say that $H$ is \textit{local time-space integrable} (with respect to the semimartingale $X$).

\begin{definition} \label{densityexp}
Assume that $H \! :[0,T] \times \rtor$ is left continuous in each argument when the other is fixed. Further assume that $t \to H(t,a)$  admits a density $h$ with respect to some Radon reference measure $\mu$. That is 
\begin{equation}
H(t,a) - H(s,a) = \spint{s}{t-}{h(u,a)}{\mu(u)},
\end{equation}
for each $s,t \in [0,T]$ and $a \in \mathbb{R}$. Let $h$ be also left continuous in $a$ for each fixed $u$. Then we define the local time space integral $ \int_{0}^{T} \spint{\mathbb{R}}{}{H(s,x)}{\lt{x}{s}}$ of $H$ to be 
\begin{equation}  \label{densityexpeqn}
\spint{\mathbb{R}}{}{H(T,a)}{_a \lt{a}{T}} - \int_0^{T-} \left( \spint{\mathbb{R}}{}{ h(u,a) }{_a \lt{a}{u}} \right) \, \mathrm{d} \mu(u),
\end{equation}
where the inner integral is the Bouleau-Yor local time integral (\ref{bouleau-yor2}).
\end{definition}

\begin{definition} \label{densityexp2}
Let $H \! : [0,T] \times \rtor$ be left continuous in each argument when the other is fixed. Further, assume that $H$ admits a density $g$ in its latter argument, with respect to some Radon reference measure $\nu$. That is
\begin{equation}
H(t,y) - H(t,x) = \spint{x}{y-}{g(t,a)}{\nu(a)},
\end{equation}
for each $t \in [0,T]$ and $x,y \in \mathbb{R}$. Also let $g$ be left continuous in $t$ for each fixed $a$. Then the local time space integral $ \int_{0}^{T} \spint{\mathbb{R}}{}{H(s,x)}{\lt{x}{s}}$ of $H$ is defined by
\begin{equation} \label{iterated2}  - \int_\mathbb{R} \spint{0}{T}{g(u,a)}{_u \lt{a}{u}} \, \mathrm{d} \nu(a). \end{equation}
\end{definition}

\begin{definition} \label{densityexp3}
Let $H \! :[0,T] \times \rtor$ be left continuous in each variable when the other is fixed. Assume that H is of locally-bounded Vitali variation, meaning for each $L > 0$,
\begin{equation} \label{variationcondition}
\sup_{\pi} \left[ \,\sum_{i,j=0}^{n-1} \Big| H(t_{i+1},x_{j+1}) - H(t_{i},x_{j+1}) - H(t_{i+1},x_{j}) + H(t_{i},x_{j}) \Big| \right]< \infty,
\end{equation}
where the supremum is taken over all finite disjoint collections of rectangles $\pi$ with vertices $(t_i\,,x_j)$ which cover $[0,T] \times [-L,L]$. Further, assume that the map $a \mapsto H(0,a)$ is of locally bounded one-dimensional variation.
We define the local time space integral $ \int_{0}^{T} \spint{\mathbb{R}}{}{H(s,x)}{\lt{x}{s}}$ of $H$ as
\begin{equation} \label{densityexp3eqn}
- \spint{\mathbb{R}}{}{\lt{a}{T} }{_a H(T,a)} - \int  \!\! \int_{[0,T] \times \mathbb{R}}\lt{a}{u}\, d_{(u,a)} H(u,a),
\end{equation}
where the final integral is a two-parameter Lebesgue-Stieltjes integral.
\end{definition}

It can easily be seen that the stipulation that $a \mapsto H(t,a)$ be of bounded variation for one fixed $t$, along with bounded two-dimensional variation, ensures that $a \mapsto H(t,a)$ is of bounded variation for all fixed $t \in [0,T]$.

The following lemma establishes the equivalence of the different representations  for the local time space integral, starting from (\ref{simpledef}). It is also useful to establish convergence of the local time-space integral of smoothed functions to develop a change-of-variables formula, Theorem \ref{seconditoform}.

\begin{lemma} \label{equiv2}
Definitions \ref{densityexp}, \ref{densityexp2} and \ref{densityexp3} agree with each other almost surely on their common domains of definition.
\end{lemma}

\begin{proof}
The proof follows by initially considering Riemann-type sums of smoothed functions, extending (\ref{simpledef}). As all the expressions agree by definition for simple functions, we take limits (in probability) to obtain the result. Uniqueness of limits in probability confirms that the definitions agree when mutually defined. 

By localisation of the underlying process, we may assume that the field of local times $(\lt{a}{t})_{a \in \mathbb{R}, 0 \leq t \leq T}$ is compactly supported. Let $[-K,K] \subset \mathbb{R}$ be a closed interval, such that $[0,T] \times [-K,K]$ contains the support of $\lt{}{}$. We take limits of integration to be consistent with left continuity of $H$ if not explicitly described. We define $H(s,a)$ for negative $s$ by $H(0,a)$, to preserve left continuity and simplify the proof. 

Mollify $H$ by taking a function $\rho \in C^\infty (\mathbb{R})$, supported on $[0,1]$, such that $\spint{\mathbb{R}}{}{\rho(x)}{x} = 1$. Define
\begin{equation} \label{mollifiedH}
\begin{split}
H^{m,n}(t,a) &= \spint{\mathbb{R}}{}{\spint{\mathbb{R}}{}{H(t- y/m,a - z/n) \, \rho(y) \, \rho(z) }{z}}{y} \\&= \spint{\mathbb{R}}{}{\spint{\mathbb{R}}{}{mn H(r,q) \, \rho(m(t-r)) \, \rho(n(a-q)) }{r}}{q},
\end{split}
\end{equation}
for each $m,n \in \mathbb{N}$. Then $H^{m,n}$ is smooth in both variables, and it is straightforward to show that $H^{m,n}$ converges to $H$ pointwise as $m,n \to \infty$ due to the left continuity of $H$. 

We fix an arbitrary partition of $[0,T] \times [-K,K]$ into rectangles, and label the vertices $(t_i,x_j)$ for $i=0,\dots,N$ and $j=0,\dots,M$. We ensure $x_0 = -K$, $x_M = K$, $t_0 = 0$ and $t_N = T$ by adding points if necessary. 

Take $\tilde{H}^{m,n} \! : [0,T] \times \mathbb{R} \to \mathbb{R}$ to be the left-endpoint approximation of $H$, namely
\begin{equation}
\tilde{H}^{m,n}(t,x) = \sum_{i=0}^{N-1} \sum_{j=0}^{M-1} H^{m,n}_{i,j} \mathbbm{1}_{(t_{i},t_{i+1}]} (t) \mathbbm{1}_{(x_{j},x_{j+1}]} (x),
\end{equation}
where $H^{m,n}_{i,j} = H^{m,n}(t_i,x_j)$. By definition, the local time-space integral of $\tilde{H}^{m,n}$ is 
\begin{equation} \label{firstsumexp}
 \Lambda(\tilde{H}^{m,n}) =  \sum_{i=0}^{N-1} \sum_{j=0}^{M-1} H^{m,n}_{i,j} \left[ \lt{x_{j+1}}{t_{i+1}} -  \lt{x_{j+1}}{t_{i}} -  \lt{x_{j}}{t_{i+1}} +  \lt{x_{j}}{t_{i}} \right].
\end{equation}
Note that this expression agrees with (\ref{simpledef}). By rearranging this sum in the style of a discrete integration by parts, we may express (\ref{firstsumexp}) in any of the following three ways:
\begin{equation} \label{midrep1}
\begin{split}
\Lambda(\tilde{H}^{m,n}) =  &\sum_{i=0}^{N-1} H^{m,n}_{i,M-1} \left[ \lt{x_{M}}{t_{i+1}} -  \lt{x_{M}}{t_{i}} \right] -  \sum_{i=0}^{N-1} H^{m,n}_{i,0} \left[ \lt{x_{0}}{t_{i+1}} -  \lt{x_{0}}{t_{i}} \right]\\  
&- \sum_{i=0}^{N-1} \sum_{j=1}^{M-1} \left[ H^{m,n}_{i,j} - H^{m,n}_{i,j-1} \right] \left[  \lt{x_{j}}{t_{i+1}} -  \lt{x_{j}}{t_{i}} \right].
\end{split}
\end{equation}

\begin{equation} \label{midrep2}
\begin{split}
\Lambda(\tilde{H}^{m,n}) =  &\sum_{j=0}^{M-1} H^{m,n}_{N-1,j} \left[ \lt{x_{j+1}}{t_{N}} -  \lt{x_{j}}{t_{N}} \right] -  \sum_{j=0}^{M-1} H^{m,n}_{0,j} \left[ \lt{x_{j+1}}{t_{0}} -  \lt{x_{j}}{t_{0}} \right] \\  
&- \sum_{i=1}^{N-1} \sum_{j=0}^{M-1} \left[ H^{m,n}_{i,j} - H^{m,n}_{i-1,j} \right] \left[  \lt{x_{j+1}}{t_{i}} -  \lt{x_{j}}{t_{i}} \right].
\end{split}
\end{equation}

\begin{equation} \label{midrep3}
\begin{split}
\Lambda(\tilde{H}^{m,n}) = \,& H^{m,n}_{N-1,M-1} \lt{x_M}{t_N} - H^{m,n}_{0,M-1} \lt{x_M}{t_0} - H^{m,n}_{N-1,0} \lt{x_0}{t_N} + H^{m,n}_{0,0} \lt{x_0}{t_0} \\ 
&- \sum_{i=1}^{N-1} \lt{x_M}{t_i} \left[ H^{m,n}_{i,M-1} - H^{m,n}_{i-1,M-1} \right] +  \sum_{i=1}^{N-1} \lt{x_0}{t_i} \left[ H^{m,n}_{i,0} - H^{m,n}_{i-1,0} \right]   \\  
&- \sum_{j=1}^{M-1} \lt{x_j}{t_N} \left[ H^{m,n}_{N-1,j} - H^{m,n}_{N-1,j-1} \right] + \sum_{j=1}^{M-1} \lt{x_j}{t_0} \left[ H^{m,n}_{0,j} - H^{m,n}_{0,j-1} \right]   \\  
&+ \sum_{i=1}^{N-1} \sum_{j=1}^{M-1} \lt{x_j}{t_i} \left[ H^{m,n}_{i,j} - H^{m,n}_{i-1,j} - H^{m,n}_{i,j-1} + H^{m,n}_{i-1,j-1} \right].
\end{split}
\end{equation}

Note that local times at $x_M$ and $x_0$ vanish by compactness of support.

We deal first with Definitions \ref{densityexp} and \ref{densityexp2}. Assume, for now, that $H$ satisfies the conditions of Definition \ref{densityexp}. By standard manipulations,
\begin{equation}
\begin{split}
H^{m,n}(t,a) - H^{m,n}(s,a) &= \spint{\mathbb{R}}{}{\spint{\mathbb{R}}{}{mn \, H(r,q) \, \Big[ \rho(m(t-r)) \\ & \hspace{120pt} - \rho(m(s-r)) \Big] \, \rho(n(a-q)) }{r}}{q} \\ &=  \spint{\mathbb{R}}{}{\spint{\mathbb{R}}{}{mn \, H(r,q) \, \spint{s}{t}{m\rho'(m(u-r))}{u} \, \rho(n(a-q)) }{r}}{q} \\ &= \spint{s}{t}{\spint{\mathbb{R}}{}{\spint{\mathbb{R}}{}{m^2 n \, H(r,q) \, \rho'(m(u-r)) \, \rho(n(a-q)) }{r}}{q}}{u}.
\end{split}
\end{equation}
Now applying properties of Lebesgue-Stieltjes integrals, 
\begin{equation}
\begin{split}
H^{m,n}(t,a) - H^{m,n}(s,a) &= \spint{s}{t}{\spint{\mathbb{R}}{}{\spint{\mathbb{R}}{}{-mn \, H(r,q) \, \rho(n(a-q)) }{_r \rho(m(u-r)) }}{q}}{u} \\ &=  \spint{s}{t}{\spint{\mathbb{R}}{}{\spint{\mathbb{R}}{}{mn \, \rho(m(u-r)) \, \rho(n(a-q)) }{_r H(r,q) }}{q}}{u} \\ &=  \spint{s}{t}{\spint{\mathbb{R}}{}{\spint{\mathbb{R}}{}{mn \, h(r,q) \, \rho(m(u-r)) \, \rho(n(a-q)) }{ \mu(r) }}{q}}{u},
\end{split}
\end{equation}
where for consistency we have chosen $h(r,q) = h(0,q)$ for negative $r$, and $\mu( (-\infty,0) ) = 0$.

We now follow the same steps when $H$ instead obeys Definition \ref{densityexp2}, which yields
\begin{equation} \label{firstpartfirstproof}
\begin{split}
H^{m,n}(t,y) - H^{m,n}(t,x) =  \spint{x}{y}{\spint{\mathbb{R}}{}{\spint{\mathbb{R}}{}{mn \, g(r,q) \, \rho(m(t-r)) \, \rho(n(p-q)) }{ \nu(q) }}{r}}{p}.
\end{split}
\end{equation}

Again, we proceed in the situation of Definition \ref{densityexp}. Following from (\ref{midrep2}), 
\begin{equation}
\begin{split}
\Lambda(\tilde{H}^{m,n}) =  &\sum_{j=0}^{M-1} H^{m,n}_{N-1,j} \left[ \lt{x_{j+1}}{t_{N}} -  \lt{x_{j}}{t_{N}} \right] -  \sum_{j=0}^{M-1} H^{m,n}_{0,j} \left[ \lt{x_{j+1}}{t_{0}} -  \lt{x_{j}}{t_{0}} \right] \\  
&- \sum_{i=1}^{N-1} \sum_{j=0}^{M-1} \spint{t_{i-1}}{t_i}{\spint{\mathbb{R}}{}{\spint{\mathbb{R}}{}{mn \, h(r,q) \, \rho(m(u-r)) \, \\& \hspace{150pt} \rho(n(x_j-q)) }{ \mu(r) }}{q}}{u}  \left[  \lt{x_{j+1}}{t_{i}} -  \lt{x_{j}}{t_{i}} \right].
\end{split}
\end{equation}
Incorporating indicator functions, we see
\begin{equation}
\begin{split}
\Lambda(\tilde{H}^{m,n}) =  &\sum_{j=0}^{M-1} \spint{\mathbb{R}}{}{H^{m,n}_{N-1,j} \indic{x_j < a \leq x_{j+1} }}{_a \lt{a}{t_N}} \\&- \sum_{i=1}^{N-1} \sum_{j=0}^{M-1} \spint{0}{T}{\spint{\mathbb{R}}{}{\spint{\mathbb{R}}{}{\spint{\mathbb{R}}{}{mn \, h(r,q) \, \rho(m(u-r)) \,\\ & \hspace{100pt} \rho(n(x_j-q)) \indic{t_{i-1} \leq u < t_i} \, \indic{x_j < a \leq x_{j+1} }}{ \mu(r) }}{q}}{_a \lt{a}{t_i}}}{u} .
\end{split}
\end{equation}

Letting the mesh of the time partition tend to zero, using continuity of local time in the time variable, we obtain
\begin{equation}
\begin{split}
&\sum_{j=0}^{M-1}  \spint{\mathbb{R}}{}{H^{m,n}(T,x_j) \indic{x_j < a \leq x_{j+1} }}{_a \lt{a}{T}} \\&- \sum_{j=0}^{M-1} \spint{0}{T}{\spint{\mathbb{R}}{}{\spint{\mathbb{R}}{}{\spint{\mathbb{R}}{}{mn \, h(r,q) \, \rho(m(u-r)) \, \rho(n(x_j-q)) \\ & \hspace{220pt}  \indic{x_j < a \leq x_{j+1} }}{ \mu(r) }}{q}}{_a \lt{a}{u}}}{u} .
\end{split}
\end{equation}

Now let the mesh of the space partition tend to zero. By smoothness of $\rho$, and left continuity of $H$ in space,
\begin{equation} \label{convolved1}
\begin{split}
\Lambda(H^{m,n}) =  & \spint{\mathbb{R}}{}{H^{m,n}(T,a)}{_a \lt{a}{T}} \\& - \spint{0}{T}{\spint{\mathbb{R}}{}{\spint{\mathbb{R}}{}{\spint{\mathbb{R}}{}{mn \, h(r,q) \, \rho(m(u-r)) \, \rho(n(a-q)) }{_a \lt{a}{u}}}{ \mu(r) }}{q}}{u} .
\end{split}
\end{equation}

Let us state the corresponding result for $H$ satisfying Definition \ref{densityexp2}. Recall that the local time is right continuous, and pointwise convergence of $\lt{x_j}{s}$ to $\lt{a}{s}$ implies weak convergence of the measures $\mathrm{d}\lt{x_j}{s}$ to $\mathrm{d}\lt{a}{s}$ for each fixed $a$. Thus, after following the same procedure in that case, with obvious modifications, we obtain
\begin{equation} \label{convolved2}
\begin{split}
\Lambda(H^{m,n}) = - \spint{\mathbb{R}}{}{\spint{\mathbb{R}}{}{\spint{\mathbb{R}}{}{\spint{0}{T}{mn \, g(r,q) \, \rho(m(s-r)) \, \rho(n(p-q)) }{_s \lt{q}{s} }}{\nu(q)}}{r}}{p}.
\end{split}
\end{equation}
Now we take limits as $m \to \infty$, then $n \to \infty$. Making the change of variables $z = m(u-r)$ in (\ref{convolved1}), we see
\begin{equation} \label{convolved3}
\begin{split}
\Lambda(H^{m,n}) =  & \spint{\mathbb{R}}{}{H^{m,n}(T,a)}{_a \lt{a}{T}} \\&- \spint{-mr}{m(T-r)}{\spint{\mathbb{R}}{}{\spint{\mathbb{R}}{}{\spint{\mathbb{R}}{}{n \, h(r,q) \, \rho(n(a-q)) }{_a \lt{a}{r + z/m}} \, \rho(z) }{ \mu(r) }}{q}}{z} .
\end{split}
\end{equation}
Taking a limit as $m \to \infty$, we note that continuity of local time in the time variable ensures the Bouleau-Yor integrals converge. Indeed, this follows from the continuity of the underlying process $X$ and the continuity of the stochastic integral. Further, unless $0 \leq r < T$ the limits of integration both diverge in the same direction, yielding zero by compactness of the support of $\rho$. So we have
\begin{equation}
\begin{split}
\Lambda(H^{n}) =  & \spint{\mathbb{R}}{}{H^{n}(T,a)}{_a \lt{a}{T}} - \spint{\mathbb{R}}{}{\spint{0}{T}{\spint{\mathbb{R}}{}{n \, h(r,q) \, \rho(n(a-q)) }{_a \lt{a}{r}}  }{ \mu(r) }}{q} .
\end{split}
\end{equation}
Finally, changing variables as before and taking a limit as $n \to \infty$, we obtain
\begin{equation}
\begin{split}
\Lambda(H) =  & \spint{\mathbb{R}}{}{H(T,a)}{_a \lt{a}{T}} - \spint{0}{T}{\spint{\mathbb{R}}{}{  h(r,a)  }{_a \lt{a}{r}}  }{ \mu(r) } .
\end{split}
\end{equation}

Again performing the same operations with obvious modifications, we see that (\ref{convolved2}) becomes
\begin{equation}
\begin{split}
\Lambda(H) = - \spint{\mathbb{R}}{}{\spint{0}{T}{ \, g(s,q) }{_s \lt{q}{s} }}{\nu(q)}.
\end{split}
\end{equation}

Now let us deal with Definition \ref{densityexp3}. We will require Lemma \ref{uniformlemma}, which is proved separately. Beginning from (\ref{midrep3}), we see
\begin{equation} 
\begin{split}
\Lambda(\tilde{H}^{m,n}) = &- \sum_{j=1}^{M-1} \spint{\mathbb{R}}{}{ \lt{x_j}{T} \indic{x_{j-1} \leq a < x_{j}} }{_a H^{m,n}(t_{N-1},a) }    \\  
&- \sum_{i=1}^{N-1} \sum_{j=1}^{M-1} \int  \!\! \int_{[0,T] \times \mathbb{R}} \lt{x_j}{t_i} \indic{t_{i-1} \leq u \leq t_i} \indic{x_{j-1} \leq a < {x_j}} \mathrm{d}_{(u,a)} H^{m,n}(u,a).
\end{split}
\end{equation}
Exchange the order of summation and integration, then take limits as the mesh of the time and space partitions converge to zero. Using the joint right continuity of local time and the dominated convergence theorem, we see
\begin{equation} \label{smoothedintegral}
\begin{split}
\Lambda(H^{m,n}) = \spint{\mathbb{R}}{}{ \lt{a}{T} }{_a H^{m,n}(T,a) }  - \int  \!\! \int_{[0,T] \times \mathbb{R}} \lt{a}{u} \, \mathrm{d}_{(u,a)} H^{m,n}(u,a).
\end{split}
\end{equation}

Using Lemma \ref{uniformlemma}, fix $\epsilon > 0$ and choose a right-continuous step function $\psi$ such that $\| \lt{}{} - \psi \|_\infty < \epsilon$. Then
\begin{equation} \label{pwconveqn}
\begin{split}
&\left| \int  \!\! \int_{[0,T] \times \mathbb{R}} \lt{a}{u} \, \mathrm{d}_{(u,a)} H^{m,n}(u,a) - \int  \!\! \int_{[0,T] \times \mathbb{R}} \lt{a}{u} \, \mathrm{d}_{(u,a)} H(u,a) \right| \\
& \leq \left| \int  \!\! \int_{[0,T] \times \mathbb{R}} \lt{a}{u} - \psi(a,u) \, \mathrm{d}_{(u,a)} H^{m,n}(u,a)  \right| + \left| \int_{[0,T] \times \mathbb{R}} \psi(a,u) \, \mathrm{d}_{(u,a)} ( H^{m,n} - H)(u,a) \right| \\ & \hspace{15pt} + \left| \int_{[0,T] \times \mathbb{R}} \psi(a,u) - \lt{a}{u} \, \mathrm{d}_{(u,a)} H(u,a) \right| \\
& \leq \epsilon \Big[ \text{TV}(H^{m,n}) + \text{TV}(H) \Big] + \Bigg|  \sum_{i=1}^m \sum_{j=1}^n a_{i,j} \, \Big[ H^{m,n}(t_i,x_j) - H(t_i,x_j) - H^{m,n}(t_{i-1},x_j) \\  & \hspace{30pt} + H(t_{i-1},x_j)  - H^{m,n}(t_i,x_{j-1})    + H(t_i,x_{j-1})  + H^{m,n}(t_{i-1},x_{j-1}) - H(t_{i-1},x_{j-1}) \Big] \Bigg|.
\end{split}
\end{equation}
By pointwise convergence of $H^{m,n}$ to $H$, each term in this final sum can be made arbitrarily small. A simpler version of the same procedure also shows that the former integral in (\ref{smoothedintegral}) converges.

Finally, let us note that the procedure we followed yields the same result independently of the choice of sequence of partitions.
\end{proof}

The following lemma is a technical result which was needed in the previous lemma, but will also be useful later.

\begin{lemma}
\label{uniformlemma}
The field of local times admits a uniform approximation by jointly right continuous step functions. That is, for all $\epsilon > 0$, there is a function $\psi$ of the form
\begin{equation}
\psi(s,a) = \sum_{i=1}^m \sum_{j=1}^n a_{i,j} \, \mathbbm{1}_{[ t_{i-1} , t_{i} )} (s) \, \mathbbm{1}_{[ x_{j-1} , x_j ) }(a),
\end{equation}
such that $\| \lt{}{} - \psi \|_{\infty} < \epsilon$ on the domain $[0,T] \times \mathbb{R}$.
\end{lemma}

\begin{proof}
Fix $\epsilon > 0$. By localising the underlying process, we may take the field of local times $\lt{}{}$ to have compact support $[-K,K] \times [0,T] \subset \mathbb{R} \times \mathbb{R}_+$.  Let $A_\epsilon = \{ (t,a) \in [-K,K] \times [0,T] \, | \, $ there exists a right continuous step function $\psi$ such that $\| \lt{}{} - \psi \|_\infty < \epsilon$ on $[0,t] \times  [-K,a]  \}$.

As $a \mapsto \lt{a}{0}$ is right continuous with left limits (in fact the zero function), it admits a uniform approximation by step functions. Thus $[-L,L] \times \{0 \} \subset A$.

Choose such a uniform approximation with error at most $\frac{\epsilon}{2}$. Precisely, take $\phi: [-L,L] \to \mathbb{R}$ such that
\begin{equation}
\phi(a) = \sum_{i=1}^n a_i \indic{x_{i-1} \leq a < x_i},
\end{equation}
and $\| \phi - \lt{}{} \| < \frac{\epsilon}{2}$. Now we may choose $\delta > 0$ such that for all $a \in [-L,L]$ and $s \in [0,\delta)$, we have $|\lt{a}{s} - \lt{a}{0} | < \frac{\epsilon}{2}$. If this were not the case, there would be some sequence $(a_k,s_k)$ such that $0 < s_k < 1/k$ and $|\lt{a_k}{s_k} - \lt{a_k}{0}| > \frac{\epsilon}{2}$. The set $\{ a_k \}$ must have an accumulation point, and so by passing to a subsequence, we may assume $a_k \to a$ monotonically for some point $a$, meaning $(a_k,s_k) \to (a,0)$ monotonically in $a$. This contradicts the joint right continuity with left limits in space, and continuity in time, of the local time.

We may now extend $\phi$ by defining
\begin{equation}
\theta(a,s) = \sum_{i=1}^n a_i \indic{x_{i-1} \leq a < x_i} \indic{0 \leq s < \delta},
\end{equation}
and noting that $\| \theta - \lt{}{} \|_\infty \leq \| \phi - \lt{}{0} \|_\infty + \| \lt{}{} - \lt{}{0} \|_\infty < \epsilon$. Thus we know $[-L,L] \times \ [0,\delta) \subseteq A$.

Define $\mathcal{T} = \sup \, \{ \, t \; | \, $ $[-L,L] \times [0,t) \subseteq A \}$. Arguing by contradiction, assume $\mathcal{T} < T$. By the same procedure as before, using joint regularity of the local time, we may choose $\delta > 0$ such that $\| \lt{a}{\mathcal{T}} - \lt{a}{s} \| < \frac{\epsilon}{3}$, uniformly in $a$, for all $s \in (\mathcal{T} - \delta,\mathcal{T} + \delta)$. By definition of $\mathcal{T}$, we may choose some jointly right continuous step function $\theta$ such that $\| \theta - \lt{}{} \| < \frac{\epsilon}{3}$ on $[-L,L] \times [0,\mathcal{T} - \frac{\delta}{2})$. Finally, as $a \mapsto \lt{a}{\mathcal{T}}$ is regulated, we choose some step function $\phi : [-L,L] \to \mathbb{R}$ such that $\| \phi - \lt{}{\mathcal{T}} \| < \frac{\epsilon}{3}$. We extend the definition of $\theta$ by defining
\begin{equation}
\tilde{\theta}(a,s) = \left\{
\begin{array}{ll}
\theta(a,s) & 0 \leq s < \mathcal{T} - \frac{\delta}{2} \\
\phi(a) & \mathcal{T} - \frac{\delta}{2} \leq s < \mathcal{T} + \frac{\delta}{2}
\end{array} \right.
\end{equation}
Again by the triangle inequality, we may show that $\| \tilde{\theta} - \lt{}{} \| < \epsilon$. Thus we have $\mathcal{T} = T$.

Finally we may show that $[-K,K] \times [0,T] = A$ by the same procedure, as $a \mapsto \lt{a}{T}$ is regulated, and replacing $(\mathcal{T} - \delta, \mathcal{T} + \delta)$ by $(T - \delta,T]$ in the argument just given.

\end{proof}

\begin{remark}
The previous definitions and lemmas of this section hold also for random functions $H : \rplusr \times \Omega \to \mathbb{R}$ which obey the conditions pathwise almost surely. For the Bouleau-Yor integral to be well defined in Definition \ref{bouleau-yor-def}, it is sufficient that $(s,\omega) \mapsto H(s,X_s,\omega)$ be a predictable process.
\end{remark}

\begin{remark}
We may generalise Definitions \ref{densityexp} and \ref{densityexp2} to the case when $H$ can be described by the sum of a finite number of measures and densities. That is,
\begin{equation}
H(t,a) - H(s,a) = \sum_{i=1}^n \spint{s}{t-}{h^i(u,a)}{\mu^i(u)},
\end{equation}
where each $h^i$ and $\mu^i$ obey the required conditions, and the analogous extension holds for Definition \ref{densityexp2}. The final formula changes in an obvious way.
\end{remark}

\begin{remark}
We chose to use left limits and the right local time, which is right continuous in space. These may be replaced by right or symmetric limits and the left or symmetric local time, provided the corresponding conditions (mostly right continuity) are satisfied. Note that the left local time has a modification which is left continuous in space. Also note that if one uses right limits with left local time, then (\ref{densityexpeqn}) should be an integral with respect to $\lt{a-}{u}$ instead.
\end{remark}

\section{A local time-space change of variables formula}
Our aim in constructing the local time-space integral was to unify various representations of the correction term in generalisations of It\^o's formula. The following change of variables formula achieves this.

\begin{theorem} \label{seconditoform}
Let $F : \rplusrtor$, given by $(t,x) \mapsto F(t,x)$ be left continuous in $t$ and continuous in $x$ when the other argument is fixed. Also assume that $F$ admits the following:
\begin{enumerate}
\item A density in the time variable $t$ with respect to some Radon reference measure $\mu$, denoted $F_t$, which admits left limits in the space variable $x$. That is,
\begin{equation} \label{itodensity}
F(t,x) - F(s,x) = \spint{s}{t-}{F_t(u,x)}{\mu(u)},
\end{equation}
for each $s,t \in [0,T]$ and $x \in \mathbb{R}$, where $F_t(u,x-)$ exists.
\item A left-partial derivative in the space variable $x$, denoted $F_x$, which is local time-space integrable in any of the given forms.
\end{enumerate}
Then we have 
\begin{equation} \label{bigito}
\begin{split}
 F(T,X_T) - F(0,X_0) = &\spint{0}{T-}{ \!\!\! F_{t}(s,X_s -)}{\mu(s)} + \spint{0}{T}{\!F_x(s,X_s)}{X_s} - \frac12  \spint{0}{T}{\! \int_\mathbb{R} F_{x}(s,a)}{\lt{a}{s}},
\end{split}
\end{equation}
for all $T \geq 0$.
\end{theorem}
\begin{proof}
We first prove the result for functions of class $C^2$. We have that 
\begin{equation}
\begin{split}
\frac12 \int_0^T \! \! \spint{\mathbb{R}}{}{ F_{x}(s,a)}{\lt{a}{s}} = &\frac{1}{2} \left( \spint{\mathbb{R}}{}{F_x(T,a)}{_a \lt{a}{T}} - \int_0^{T} \left( \spint{\mathbb{R}}{}{ F_{xt}(s,a) }{_a \lt{a}{s}} \right) \, \mathrm{d} s \right),
\end{split}
\end{equation}
which follows from Definition \ref{densityexp}. We expand the right-hand side using (\ref{bouleau-yor2}), giving
\begin{equation}
\begin{split}
& F(T,X_T) - F(T,X_0) - \spint{0}{T}{F_x(T,X_u)}{X_u} \\ & - \int_0^{T} \Big( F_{t}(s,X_s) - F_{t}(s,X_0) - \spint{0}{s}{ F_{xt}(s,X_u) }{X_u} \Big) \, \mathrm{d} s.
\end{split}
\end{equation}
Now apply the stochastic Fubini theorem and fundamental theorem of calculus, to give
\begin{equation}
\begin{split}
&F(T,X_t) - F(0,X_0) - \spint{0}{T}{F_t(s,X_s)}{s} \\ &- \spint{0}{T}{F_x(T,X_u)}{X_u} + \spint{0}{T}{\spint{u}{T}{F_{xt}(s,X_u)}{s}}{X_u}.
\end{split}
\end{equation}
Another application of the fundamental theorem of calculus yields the result.

We then establish (\ref{bigito}) by approximating functions in the domain of definition of each representation by functions in $C^2$. Specifically, mollify $F$ in the same form as (\ref{mollifiedH}). Then we have
\begin{equation} \label{mollifiedF}
\begin{split}
F^{m,n}(t,a) &= \spint{\mathbb{R}}{}{\spint{\mathbb{R}}{}{F(t- y/m,a - z/n) \, \rho(y) \, \rho(z) }{z}}{y} \\&= \spint{\mathbb{R}}{}{\spint{\mathbb{R}}{}{mn F(r,q) \, \rho(m(t-r)) \, \rho(n(a-q)) }{r}}{q}.
\end{split}
\end{equation}
The stochastic and classical dominated convergence theorems and the proof of Lemma \ref{equiv2} give convergence of the stochastic and local time-space integrals. Only the time integral remains. We can see that
\begin{equation}
\begin{split}
\spint{0}{T}{F^{m,n}_{t}(s,X_s)}{s} &= \spint{0}{T}{ \frac{\mathrm{d}}{\mathrm{d}s} \left. \left( \spint{\mathbb{R}}{}{\spint{\mathbb{R}}{}{mn F(r,q) \, \rho(m(s-r)) \, \rho(n(a-q)) }{r}}{q} \right) \right|_{a = X_s} }{s}  \\ &= \spint{0}{T}{ \left. \left( \spint{\mathbb{R}}{}{\spint{\mathbb{R}}{}{m^2n F(r,q) \, \rho'(m(s-r)) \, \rho(n(a-q)) }{r}}{q} \right) \right|_{a = X_s} }{s}  \\ &= \spint{0}{T}{ \left. \left( \spint{\mathbb{R}}{}{\spint{\mathbb{R}}{}{-mn F(r,q)  \, \rho(n(a-q)) }{_r \rho(m(s-r))}}{q} \right) \right|_{a = X_s} }{s} \\ &= \spint{0}{T}{  \left. \left( \spint{\mathbb{R}}{}{\spint{\mathbb{R}}{}{mn F_t(r,q) \, \rho(m(s-r)) \, \rho(n(a-q)) }{ \mu(r) }}{q} \right) \right|_{a = X_s} }{s} \\ &= \spint{0}{T}{ \spint{\mathbb{R}}{}{\spint{\mathbb{R}}{}{mn F_t(r,q) \, \rho(m(s-r)) \, \rho(n(X_s -q)) }{ \mu(r) }}{q} }{s}.
\end{split}
\end{equation}
Now make the substitution $u = m(s-r)$. We get
\begin{equation}
\begin{split}
\spint{0}{T}{F^{m,n}_{t}(s,X_s)}{s} &= \spint{\mathbb{R}}{}{\spint{\mathbb{R}}{}{ \spint{-mr}{m(T-r)}{ n F_t(r,q) \, \rho(u) \, \rho(n(X_{r + u/m} -q))}{u} }{ \mu(r) }}{q}.
\end{split}
\end{equation}
Now we take a limit as $m \to \infty$. Recalling that $X$ is continuous, we also note that the limits of the $u$ integral diverge to the same limit unless $0 \leq r < T$. Thus we have
\begin{equation}
\begin{split}
\spint{0}{T}{F^{n}_{t}(s,X_s)}{s} &= \spint{\mathbb{R}}{}{\spint{0}{T}{ n F_t(r,q) \, \rho(n(X_{r} -q)) }{ \mu(r) }}{q}.
\end{split}
\end{equation}
Exchanging the order of integration, changing variables, then taking a limit as $n \to \infty$ and using the left limits of $F_t$ in space gives
\begin{equation}
\begin{split}
\spint{\mathbb{R}}{}{\spint{0}{T}{ n F_t(r,q) \, \rho(n(X_{r} -q)) }{ \mu(r) }}{q} \to \spint{0}{T}{ F_t(r,X_r -) }{ \mu(r) } .
\end{split}
\end{equation}
This completes the proof.
\end{proof}

\begin{remark} \label{itodensity2-remark}
We may generalise (\ref{itodensity}) by replacing $\mu$ and $F_t$ by a finite sum of arbitrary Radon measures and densities, which are left continuous in space, and obey the obvious analogue of (\ref{itodensity}).
\end{remark}

\section{Continuity and properties of local time-space integration}

A key consequence of the classical It\^o formula is that the semimartingale decomposition of $F(B_t)$ is given explicitly. Specifically, the quadratic variation integral is of bounded variation. The same holds for the local time-space integral, except for the integral with respect to the space variable of local time, which cannot be of bounded variation in time in general.

\begin{theorem}
If $H$ is local time-space integrable, then expression (\ref{iterated2}), both integrals in (\ref{densityexp3eqn}) and the two-variable integral in (\ref{densityexpeqn}) are of bounded variation.
\end{theorem}

\begin{proof}
Take a partition $\pi$ of $[0,T]$, such that $t_0 = 0$ and $t_N = T$. Dealing with (\ref{iterated2}),
\begin{equation}
\begin{split} 
&\sum_{t_i \in \pi} \left| \int_\mathbb{R} \spint{t_i}{t_{i+1}}{g(u,a)}{_u \lt{a}{u}} \, \mathrm{d} \nu(a) \right| \leq \int_\mathbb{R}  \sum_{t_i \in \pi}  (\lt{a}{t_{i+1}} - \lt{a}{t_i}) \sup_{t_i \leq u \leq t_{i+1}} |g(u,a)| \, \mathrm{d} |\nu|(a) \\ & \leq \int_\mathbb{R} \lt{a}{T} \sup_{0 \leq u \leq T} |g(u,a)| \, \mathrm{d} |\nu|(a).
\end{split}
\end{equation}
By continuity of the underlying semimartingale, the local time is almost surely compactly supported. Thus the right-hand integral may be taken over a compact set, and so the right hand side is finite and independent of the partition. 

The variation of (\ref{densityexp3eqn}) over any partition is bounded by the supremum of the local time, which is almost-surely bounded pathwise, and the Vitali variation of $H$.

Finally, as $h$ is bounded in (\ref{densityexpeqn}), and $\mu$ is of finite total variation on compacts, the two-variable integral is of bounded variation.
\end{proof}

It is natural to be concerned with properties of the map $H \mapsto \Lambda(H)$, meaning continuity or a type of dominated convergence theorem.

\begin{theorem} \label{ctsinintegrand}
Let $H^n$ and $H$ be a sequence of local time-space integrable functions satisfying any one of the following conditions:
\begin{enumerate}
\item $H^n$ and $H$ satisfy Definition \ref{densityexp}, and there exists a measure $\gamma$ such that $\nu^n \ll \gamma$ and $\nu \ll \gamma$, where the densities $g^n \frac{\mathrm{d}\nu^n}{\mathrm{d}\gamma}$ are uniformly locally bounded and converge pointwise to $g \frac{\mathrm{d}\nu}{\mathrm{d}\gamma}$;
\item $H^n$ and $H$ satisfy Definition \ref{densityexp2}, and there exists a measure $\gamma$ such that $\mu^n \ll \gamma$ and $\mu \ll \gamma$, where the densities $h^n \frac{\mathrm{d}\mu^n}{\mathrm{d}\gamma}$ are uniformly locally bounded and converge pointwise to $h \frac{\mathrm{d}\mu}{\mathrm{d}\gamma}$;
\item $H^n$ and $H$ satisfy Definition \ref{densityexp3}, with uniformly locally-bounded total variation, and $H^n \to H$ pointwise.
\end{enumerate}
Then we have
\begin{equation}
\int_\mathbb{R} \, \spint{0}{T}{H^n(s,a)}{\lt{a}{s}(X)} \to \int_\mathbb{R} \, \spint{0}{T}{H(s,a)}{\lt{a}{s}(X)},
\end{equation}
where convergence holds uniformly on compacts in probability.
\end{theorem}

\begin{proof}
The final case follows in the same way as the final part of the proof of Lemma \ref{equiv2}, namely (\ref{pwconveqn}) with $H^{m,n}$ there replaced by $H^n$.
The other two cases are straightforward applications of the deterministic and stochastic dominated convergence theorems.
\end{proof}

We are motivated by \cite{protter93} to also consider some form of continuity depending on the underlying semimartingale. Let us first introduce a type of localisation. Given a sequence of semimartingales $(X^n)_{n \geq 1}$, we will refer to the bounded variation and local martingale components of $X^n$ as $A^n$ and $M^n$ respectively (which vanish at time zero almost surely). We define $\mathcal{H}^1$ to be the set of continuous semimartingales $X$, with decomposition $X_0 + M+A$, such that $\| X \|_{\mathcal{H}^1} = \mathbb{E} \left [ \left< M,M \right>_\infty^{1/2} + \int_0^\infty  | \mathrm{d} A_s | \, \right]$ is finite.

\begin{definition}  \label{localh1}
We say that a sequence of semimartingales $\left( X^n \right)_{n \geq 1}$ are uniformly locally in $\mathcal{H}^1$ if there exists a sequence of stopping times $T_m$, increasing to infinity almost surely, such that for all $n \in \mathbb{N}$ we have $\| X^n_{\, \cdot \,  \land T_m} \|_{\mathcal{H}^1} \leq K(m)$ for some constant $K(m)$ independent of $n$. \end{definition}

Note that all continuous semimartingales are locally in $\mathcal{H}^1$. Asking that a sequence be uniformly locally in $\mathcal{H}^1$ prevents the following situation. Take $X^n$ to be a Brownian motion starting at zero, reflected between barriers at $0$ and $1/n$. The quadratic variation of each $X^n$ is independent of $n$, and as $n \to \infty$ the $X^n$ converge uniformly on compacts in probability to the zero process. One may see that $\int_0^t  | \mathrm{d} A^n_s |$ is unbounded in $n$ for any fixed $t$, so the sequence does not obey Definition \ref{localh1}. In general, however, we have the following result, due to Barlow and Protter \cite{barlow90}.

\begin{lemma}[\cite{barlow90} Thm 1, Corr 2] \label{barlowresult}
Let $X^n$ be a sequence of semimartingales in $\mathcal{H}^1$,  such that $\spint{0}{\infty}{| \!}{A^n_s|} \leq K$ for some non-random constant $K$. Assume that there is some process $X$ such that $\mathbb{E} \left[ (X^n - X)^* \right] \to 0$. Then $X$ is a semimartingale, and both $\lim_{n \to \infty} \| (M^n - M)^* \|_{\mathcal{H}^1} = 0$ and $\mathbb{E} \left[ (A^n - A)^* \right] \to 0$, with $\spint{0}{\infty}{| \!}{A_s|} \leq K$.
\end{lemma}

We now impose stronger conditions on a sequence of semimartingales.

\begin{definition}
A sequence of semimartingales $X^n$ will be called admissible if;
\begin{enumerate}
\item For each $n\in \mathbb{N}$, $X^n_0 = 0$ almost surely.
\item There exists a sequence of stopping times $T_m$ such that $T_m \uparrow \infty$ almost surely, and
\begin{equation}
\sup_{s \in \mathbb{R}_+} \left| X^n_{T_m \land s} \right| < K(m)
\end{equation}
almost surely, for some constant $K(m)$ independent of $n$.
\item The sequence $X^n$ is locally uniformly in $\mathcal{H}^1$.
\end{enumerate}
\end{definition}

We now determine when convergence of an admissible sequence implies convergence of the associated local time-space integrals.

\begin{theorem} \label{ctsinsemimg}
Let $X^n$ be a sequence of admissible semimartingales converging to a semimartingale $X$ uniformly on compacts in probability. Let $H$ be local time-space integrable in the sense of Definition \ref{densityexp2} or Definition \ref{densityexp3}. Assume that $\left| \mathrm{d} A^n \right|$ converges weakly in probability to some (random) measure $\lambda$, meaning
\begin{equation}
\spint{0}{t}{f(s) |\!}{A^n_s|} \to \spint{0}{t}{f(s) }{\lambda(s)}
\end{equation}
in probability, for each continuous function $f: \mathbb{R}_+ \to \mathbb{R}$ and $t \in \mathbb{R}_+$. If for all  $a \in \mathbb{R}$
\begin{equation} \label{ctsassump}
\mathbb{E} \left[ \spint{0}{t}{\indic{X_s = a} }{\lambda(s)} \right] = 0
\end{equation} 
for all $t \in \mathbb{R}_+$, then we have
\begin{equation}
\int_\mathbb{R} \, \spint{0}{T}{H(s,a)}{\lt{a}{s}(X^n)} \to \int_\mathbb{R} \, \spint{0}{T}{H(s,a)}{\lt{a}{s}(X)},
\end{equation}
where convergence holds uniformly on compacts in probability.
\end{theorem}

\begin{proof}
Take a sequence of stopping times $T_m$ with $T_m \uparrow \infty$ almost surely, such that
\begin{equation}
\sup_{s \in \mathbb{R}+} \left| X^n_{s \wedge T_m} \right| < K(m)
\hspace{15pt} \text{ and } \hspace{15pt}
\| X^n_{\cdot \, \wedge T_m} \|_{\mathcal{H}^1} < K(m)
\end{equation} 
almost surely for some $K(m)$ independent of $n$. Fix $T \in \mathbb{R}_+$ and take $M \in \mathbb{N}$ such that $\mathbb{P} \left( T_m < T \right) < \epsilon$ for all $m \geq M$. On the complement of this set, that is on $\left\{T_m \geq T \right\}$, $\lt{a}{t}(X^n)$ is compactly supported in $[0,T] \times [-K(M),K(M)]$. Again on $\left\{T_m \geq T \right\}$, using the Tanaka formula and triangle inequality, we can show
\begin{equation}
C = \sup_n \left( \sup_{[0,T] \times [-K,K]} \mathbb{E} \left[ \lt{a}{t}(X^n) \right] \right) < \infty.
\end{equation}
It follows that the same bound holds for $X$ in place of $X^n$ on the set $\left\{T_m \geq T\right\}$ with the same constant $C$ by Lemma \ref{barlowresult}.

Beginning with Definition \ref{densityexp2}, by elementary bounds we have
\begin{equation} \label{originalexpression}
\begin{split}
&\left|\int_\mathbb{R} \, \spint{0}{T}{g(s,a)}{_s \lt{a}{s}(X^n)} \, \mathrm{d} \nu(a) - \int_\mathbb{R} \, \spint{0}{T}{g(s,a)}{_s \lt{a}{s}(X)} \, \mathrm{d} \nu(a) \right| \\
&\leq \spint{\mathbb{R}}{}{\left| \int_0^T \! g(s,a) \, \mathrm{d}_s (\lt{a}{s}(X^n) - \lt{a}{s}(X))\right| }{|\nu|(a)} .
\end{split}
\end{equation}
Restricting to the set $\left\{ T_m \geq T \right\}$, take an expectation and apply Fubini's theorem. We then determine the limit as $n \to \infty$ of 
\begin{equation} \label{expectationexpression}
\mathbb{E} \left[ \left| \int_0^T \! g(s,a) \, \mathrm{d}_s (\lt{a}{s}(X^n) - \lt{a}{s}(X))\right| \indic{T_m \geq T}\right],
\end{equation} 
for each fixed $a$. Fix $\delta > 0$. As $g$ is regulated in $s$, we may choose some left continuous step function $\psi$ such that $\| g(\cdot,a) - \psi \| < \delta/2C$. Using the triangle inequality, we may bound (\ref{expectationexpression}) by
\begin{equation} 
\begin{split}
&\mathbb{E} \left[ \left| \int_0^T \! \big( g(s,a) - \psi(s) \big) \, \mathrm{d}_s (\lt{a}{s}(X^n) - \lt{a}{s}(X))\right| \indic{T_m \geq T} \right] \\ &+ \mathbb{E} \left[ \left| \int_0^T \!  \psi(s) \, \mathrm{d}_s (\lt{a}{s}(X^n) - \lt{a}{s}(X))\right| \indic{T_m \geq T} \right].
\end{split}
\end{equation} 
The former term is bounded by $2 \delta$. We use the following Lemma \ref{convergenceoflt} to see that the latter term is bounded and converges to zero in probability as $n \to \infty$, giving convergence in expectation. Now we may apply the dominated convergence theorem to the expression
\begin{equation}
\spint{\mathbb{R}}{}{\mathbb{E} \left[ \left| \int_0^T \! g(s,a) \, \mathrm{d}_s (\lt{a}{s}(X^n) - \lt{a}{s}(X))\right| \indic{T_m \geq T}\right] }{|\nu|(a)},
\end{equation}
to obtain convergence to zero. This gives convergence in probability of (\ref{originalexpression}).

A more straightforward version of the same method gives the result for Definition \ref{densityexp3}.
\end{proof}

Now we establish conditions for convergence in probability of the local times. Afterwards, we shall consider ways to verify the assumptions of the following lemma which are useful for us, but by no means exhaustive. 
\begin{lemma} \label{convergenceoflt}
In the setting of Theorem \ref{ctsinsemimg}, we have $\lt{a}{s}(X^n) \to \lt{a}{s}(X)$ uniformly on compacts in probability in $s$, for each fixed $a \in \mathbb{R}$.
\end{lemma}

\begin{proof}
First let us assume that $X^n$ and $X$ are uniformly bounded in $\mathcal{H}^1$, with $\mathbb{E} \left[ (X^n - X)^* \right] \to 0$. This gives that $\mathbb{E} \left[ \left< M^n -M, M^n - M \right>_\infty^\frac{1}{2} \right] \to 0$ and $\mathrm{d}A^n \to \mathrm{d}A$ weakly in expectation. Furthermore, $\mathrm{d}A^n$ and $\mathrm{d}A$ have total variation uniformly bounded by some deterministic 
constant $K$. 

Expanding via the Tanaka formula, we see
\begin{equation}
\begin{split}
&|\lt{a}{t}(X^n) - \lt{a}{t}(X)| \\
& = \Big| \, \Big( |X^n_t - a| - |X^n_0 - a| - \spint{0}{t}{\text{sgn}(X^n_s - a)}{X^n_s} \Big ) \\ & \hspace{15pt} - \Big( |X_t - a| - |X_0 - a|  - \spint{0}{t}{\text{sgn}(X_s - a)}{X_s} \Big) \Big| \\
&\leq \Big| |X^n_t - a| - |X_t - a| \Big| + \Big| |X^n_0 - a| - |X_0 - a| \Big| \\
& \hspace{15pt}+ \left| \spint{0}{t}{\text{sgn}(X^n_s - a)}{X^n_s} - \spint{0}{t}{\text{sgn}(X_s - a)}{X_s} \right|
\end{split}
\end{equation}
The first two terms converge to zero uniformly on compacts in in probability by uniform convergence in expectation of $X^n$ to $X$, and the continuous mapping theorem. Using the decompositions of $X^n$ and $X$, the difference of the stochastic integrals with respect to the local martingale parts is bounded by
\begin{equation} \label{triangleineq}
\begin{split}
&| \spint{0}{t}{\text{sgn}(X^n_s - a)}{M^n_s} - \spint{0}{t}{\text{sgn}(X^n_s - a)}{M_s} | \\
&+ | \spint{0}{t}{\text{sgn}(X^n_s - a)}{M_s} - \spint{0}{t}{\text{sgn}(X_s - a)}{M_s} |.
\end{split}
\end{equation}
Dealing with each part separately, we have
\begin{equation} \label{triangleineq2}
\begin{split}
&\mathbb{E} \left[ | \spint{0}{t}{\text{sgn}(X^n_s - a)}{M^n_s} - \spint{0}{t}{\text{sgn}(X^n_s - a)}{M_s} | \right]  \\
&= \mathbb{E} \left[ \left( \spint{0}{t}{\text{sgn}(X^n_s - a)^2}{ \!\left< M^n - M,M^n-M \right>_s} \right)^{\frac12} \right] \\ &\leq \mathbb{E} \left[ \left( \spint{0}{t}{}{ \!\left< M^n - M,M^n-M \right>_s} \right)^{\frac12}\right] \\ &= \mathbb{E}[\left< M^n - M, M^n - M \right>_t^{\frac12}] \leq \mathbb{E}[\left< M^n - M, M^n - M \right>_{\infty}^{\frac12}].
\end{split}
\end{equation}
This converges to zero by assumption, and is independent of $t$.

Expanding the other part, fixing some $\epsilon > 0$, we see
\begin{equation} \label{triangleineq3}
\begin{split}
&\mathbb{E} \left[ | \spint{0}{t}{\text{sgn}(X^n_s - a)}{M_s} - \spint{0}{t}{\text{sgn}(X_s - a)}{M_s} |^2 \right] \\ &= \mathbb{E} \left[ \spint{0}{t}{\left( \text{sgn}(X^n_s - a) - \text{sgn}(X_s - a) \right)^2 }{\!\left< M,M \right>_s} \right] \\ & \leq \mathbb{E} \bigg[ \int_{0}^{t}{ \left( \text{sgn}(X^n_s - a) - \text{sgn}(X_s - a) \right)^2 \left( \indic{X_s \notin (a-\epsilon,a+\epsilon)}\right.}  \\& \hspace{200pt} +  {\left.\indic{X_s \in (a-\epsilon,a+\epsilon)} \right)\, }{\mathrm{d}\! \left< M,M \right>_s} \bigg] \\
\end{split}
\end{equation}
\begin{equation*}
\begin{split}
&\leq \mathbb{E} \left[ \spint{0}{t}{ \left( \text{sgn}(X^n_s - a) - \text{sgn}(X_s - a) \right)^2 \indic{X_s \notin (a-\epsilon,a+\epsilon)} }{\!\left< M,M \right>_s} \right] \\ & \hspace{15pt} + \mathbb{E} \left[ \spint{0}{t}{ \indic{X_s \in (a-\epsilon,a+\epsilon)} }{\!\left< M,M \right>_s} \right].
\end{split}
\end{equation*}
For each fixed $\epsilon > 0$, as $n \to \infty$ the former term converges to zero by uniform convergence of $X^n$ to $X$ in expectation. The latter term may be re-expressed using the occupation time formula (\ref{otf}) as
\begin{equation}
\mathbb{E} \left[ \spint{\mathbb{R}}{}{\indic{x \in (a-\epsilon,a+\epsilon)} \lt{x}{t} }{x} \right].
\end{equation}
Letting now $\epsilon \to 0$, this converges to zero. Uniformity in $t$ follows after noting that the integrands are positive, and the measures $\mathrm{d} \!\left< M,M \right>$ are non-negative.

We follow the same procedure for the Lebesgue-Stieltjes integrals. Let $U_0 = \{ (s,\omega) \in \mathbb{R}_+ \times \Omega \, : \, | X_s - a | \geq 1 \}$ and $U_i = \{ (s,\omega) \in \mathbb{R}_+ \times \Omega \, : \, 2^{-(i+1)} \leq | X_s - a | < 2^{-i} \}$ for each $i \in \mathbb{N}$. Employing Tonelli's theorem,
\begin{equation} \label{triangleineq4}
\begin{split}
&\mathbb{E} \left[ | \spint{0}{t}{\text{sgn}(X^n_s - a)}{A^n_s} - \spint{0}{t}{\text{sgn}(X_s - a)}{A^n_s} | \right]  \\ & = \mathbb{E} \left[ |  \spint{0}{t}{\sum_{i=0}^\infty \Big( \text{sgn}(X^n_s - a) - \text{sgn}(X_s - a) \Big) \mathbbm{1}_{U_i} }{A^n_s} | \right] \\ & \hspace{10pt} + \mathbb{E} \left[ |  \spint{0}{t}{ \Big( \text{sgn}(X^n_s - a) - \text{sgn}(X_s - a) \Big) \indic{X_s = a} }{A^n_s} | \right]
\\ &\leq \sum_{i=0}^\infty \mathbb{E} \left[ \, \sup_{0 \leq s \leq t} \Big| \text{sgn}(X^n_s - a) - \text{sgn}(X_s - a) \, \Big| \mathbbm{1}_{U_i}  \, \text{TV} \left( \mathrm{d} A^n_s \, |_{U_i} \right) \, \right]
\\& \hspace{15pt} + \mathbb{E}  \left[ | \spint{0}{t}{ \Big( \text{sgn}(X^n_s - a) - \text{sgn}(X_s - a) \Big) \indic{X_s = a}}{A^n_s} | \right]
\end{split}
\end{equation}
Uniform convergence of $X^n$ to $X$ in expectation implies that the sum converges to zero as $n \to \infty$, after employing the dominated convergence theorem. Again this is uniform in $t$. The final term is zero by assumption (\ref{ctsassump}) for each $t$.
Finally,
\begin{equation}
\begin{split}
&\mathbb{E} \left[ | \spint{0}{t}{\text{sgn}(X_s - a)}{A^n_s} - \spint{0}{t}{\text{sgn}(X_s - a)}{A_s} | \right] \to 0
\end{split}
\end{equation}
as the set of possible discontinuity points $\left\{ X_s = a \right\}$ is $\lambda$-null, again by (\ref{ctsassump}). One can also observe this by writing the signum function as a difference of two indicator functions, and considering measures of the sets $\{ X_s \geq a \}$ and $\{ X_s < a \}$. This is independent of $t$ by assumption.

Finally, let $X^n$ and $X$ be as in the statement of the theorem. Then there exists a sequence of stopping times $T_m$ which increase to infinity almost surely, such that $\| X^n_{\cdot \, \wedge T_m} \|_{\mathcal{H}^1} < K(m)$, and $\| X_{\cdot\, \wedge T_m} \|_{\mathcal{H}^1} < K(m)$, with each of the stopped processes bounded. Fix $t_0 > 0$ and positive constants $\epsilon$ and $\delta $. Take $M$ such that $\mathbb{P} \left( T_M \leq t_0 \right) < \epsilon$. Then it follows
\begin{equation}
\begin{split}
&\mathbb{P} \left( \sup_{0 \leq t \leq t_0} |\lt{a}{t}(X^n) - \lt{a}{t}(X)| > \delta \right) \\ &\leq \mathbb{P} \left( \sup_{0 \leq t \leq t_0} |\lt{a}{t}(X^n_{\cdot \, \wedge T_m}) - \lt{a}{t}(X_{\cdot\, \wedge T_m})| > \delta \right) + \mathbb{P} \left( T_M \leq t_0 \right).
\end{split}
\end{equation}
Applying the former part of this proof to the stopped semimartingales $X^n_{\cdot \, \wedge T_m}$ and $X_{\cdot\, \wedge T_m}$, noting that convergence in probability and uniform boundedness implies convergence in expectation, we obtain the result.
\end{proof}

The next result is a probabilistic version of a general result in the theory of functions of bounded variation, namely that convergence in $L^1$ and convergence of total variations in $\mathbb{R}$ implies weak convergence. This helps to verify (\ref{ctsassump}) when $\lambda = \left| \mathrm{d} A_s \right|$. 

\begin{lemma} \label{totalvarconv}
Let $A^n$ be a sequence of processes of locally uniformly bounded variation, meaning that there exists a sequence of stopping times $T^m$ increasing to $\infty$ almost surely, such that $\spint{0}{\infty}{|\!}{ A_{s \wedge T_m}^n |} < K(m)$ for constants $K(m)$. Further assume there exists a process of locally bounded variation $A$ such that $\spint{0}{t}{\left| (A^n_s - A_s) \right|}{s} \to 0 $ in probability for each $t \in \mathbb{R}_+$. If we have $\spint{0}{t}{| \!}{A^n_s|} \to \spint{0}{t}{| \!}{A_s|}$ in probability for each $t \in \mathbb{R}_+$, then also $| \mathrm{d}A^n| \to | \mathrm{d}A |$ weakly in probability.
\end{lemma}
\begin{proof}
By taking a minimum, we assume that $T^m$ also localises $\spint{0}{\infty}{| \!}{A_s |}$ with the same constant $K(m)$. Fix $\epsilon > 0$ and $t_0 > 0$ and take $M$ such that $\mathbb{P} \left( T_M \leq t_0 \right) < \epsilon$. Fixing a continuous $f$, by straightforward bounds we find
\begin{equation}
\mathbb{E} \left[ \left| \spint{0}{t_0}{f(s) |\!}{A^n_{s \wedge T_M} |} - \spint{0}{t_0}{f(s) |\!}{A_{s \wedge T_M}|} \, \right|^2 \right] \leq 4 \, K(M) \sup_{0 \leq s \leq t_0} \left| f(s) \right|^2,
\end{equation}
meaning that this sequence of random variables is uniformly integrable and lies in $L^1(\mathbb{P})$. Thus it is relatively sequentially compact in the weak topology by the Dunford-Pettis theorem. Given any convergent subsequence, by passing to a further subsequence, we obtain a subsequence converging almost surely. The deterministic Theorem \cite[Prop 3.15, p. 126]{ambrosio00}  applies to this subsequence, which therefore converges to zero in expectation. As any subsequence has a convergent subsequence with the same limit, we obtain that the original sequence converges to zero in the weak topology. By using the bounded test function $\indic{\Omega} (\omega)$ we obtain the result.
\end{proof}

\begin{remark}
Note that condition (\ref{ctsassump}) is satisfied when $\lambda$ admits a density with respect to Lebesgue measure almost surely, and $\mathbb{P} \left( X_{s} = a \right) = 0$ for each $0 < s \leq t$, by applying Fubini's theorem. In particular, this holds in the important special case when $X$ is a one-dimensional It\^o process with strictly positive diffusion coefficient.
\end{remark}

\section{Applications}
The following theorem, giving the local time on curves formula, follows as a direct consequence of the local time-space calculus, namely Theorem \ref{seconditoform}. The formula was first established by Peskir \cite{peskir05}, dealing with a function which is sufficiently smooth except over a time-dependent curve, and was extended to higher dimensions and to include processes with jumps in \cite{peskir07}. Here we deal with the original time-space case.

\begin{theorem} \label{ltcformula}
Let $X = (X_t)_{t \geq 0} $ be a continuous semimartingale, with corresponding field of local times $\lt{a}{s}$, and let $b: \rplustor$ be a continuous function of bounded variation. Define
\begin{gather*}
 C = \{ (t,x) \in \rplusr \; | \; x < b(t) \}, \\
 D = \{ (t,x) \in \rplusr \; | \; x > b(t) \}.
\end{gather*}
Suppose we are given a continuous function $F: \rplusrtor$ such that
\begin{gather*}
F \text{ is } C^{1,2} \text{ on } \bar{C}, \\
F \text{ is } C^{1,2} \text{ on } \bar{D}.
\end{gather*}
The precise meaning of this condition is that the restriction of $F$ to $C$ can be extended to a $C^{1,2}$ function on the whole of $\rplusr$, and likewise for $D$. Then the following change of variable formula holds
\begin{equation}
\begin{split}
F(t,X_t) = F(0,X_0) &+  \spint{0}{t}{F_t(s,X_s-) }{s}
+ \spint{0}{t}{F_x(s,X_s-)}{X_s} \\ 
&+ \frac{1}{2} \spint{0}{t}{ F_{xx}(s,X_s) \,  \indic{X_s \neq b(s)}}{\! \left<X,X\right>_s} \\ 
& +  \spint{0}{t}{\frac{1}{2} \Big( F_x(s,b(s)+) - F_x(s,b(s)-) \, \Big) \indic{X_s = b(s)}}{_s \lt{b(s)}{s}}.
\end{split}
\end{equation}
Equivalently, we may express this as
\begin{equation} \label{ltcsecond}
\begin{split}
F(t,X_t) = F(0,X_0) &+  \spint{0}{t}{F_t(s,X_s -) }{s} +  \spint{0}{t}{F_x(s,X_s-)}{X_s} \\ 
&+ \frac{1}{2} \spint{\mathbb{R}}{}{\spint{0}{t}{ F_{xx}(s,a + b(s)) \,  \indic{a \neq b(s)}}{_s \lt{b(s) + a}{s}}}{a} \\ 
&+  \spint{0}{t}{\frac{1}{2} \Big( F_x(s,b(s)+) - F_x(s,b(s)-) \, \Big) \indic{X_s = b(s)}}{_s \lt{b(s)}{s}},
\end{split}
\end{equation}
where $\lt{b(s) + a}{s}$ is the local time of $X-b$ at $a$.
\end{theorem}
\begin{proof}
The equivalence of the representations follows immediately from the time-dependent occupation time formula (\ref{otf}). Write $G(s,x) = F(s, x + b(s))$, and note that $Y = X - b$ is a semimartingale as $b$ is of bounded variation. Then $F(s,X_s) = G(s,Y_s)$ by definition.

We can see that $G$ admits a second space derivative as a density with respect to the measure $\nu$, given by
\begin{equation}
\nu(A) = \delta_0 (A) + \text{Leb}(A),
\end{equation}
where $\delta_0$ is the Dirac mass at $0$, and
\begin{equation}
G_{xx}(s,a) = \left\{ 
\begin{array}{ll}
F_x(s,b(s)+) - F_x(s,b(s)-) & {a = 0}, \\
F_{xx}(s,a + b(s)) &  a \neq 0.
\end{array} \right.
\end{equation}
This means $G_x$ is local time-space integrable in the form of Definition \ref{densityexp2}. As $b$ is of bounded variation, it generates a Lebesgue-Stieltjes measure $\mathrm{d}b$. Denoting by $\mathrm{d}b^c$ and $\mathrm{d}b^\bot$ the Lebesgue-continuous and Lebesgue-singular part of $\mathrm{d}b$ respectively, we have by the Lebesgue-Stieltjes chain rule
\begin{equation}
\begin{split}
G(t,x) = \, & G(s, x)  + \spint{s}{t}{F_t(u, x + b(u)-) + F_x(u,x + b(u) -) \, \frac{\mathrm{d}b^c}{\mathrm{d}u}}{u} \\ \, &+ \spint{s}{t}{F_x(u, x + b(u) -)}{b^\bot (\mathrm{d}u)}.
\end{split}
\end{equation}
So we have something as described in Remark \ref{itodensity2-remark}, with left-continuous integrands in space.

Now apply Theorem \ref{seconditoform} to $G$ and $Y$. Substituting in the expressions of $G$ and its derivatives in terms of $F$, we obtain
\begin{equation}
\begin{split}
F(T,X_T) - F(0,X_0) = &\spint{0}{T}{F_t(s, X_s -)}{s} + \spint{0}{T}{F_x(s,X_s -)}{b(s)} \\ &+ \spint{0}{T}{F_x(s,X_s-)}{(X - b)_s} \\ &- \frac12 \int_\mathbb{R} \spint{0}{T}{F_{xx}(s,a + b(s)) \indic{a \neq 0}}{_s \lt{a}{s}(Y)} \, \mathrm{d}a \\ & + \spint{0}{T}{\left( F_x(s,b(s)+) - F_x(s,b(s)-) \right)}{_s \lt{0}{s}(Y)}.
\end{split}
\end{equation}
where the expression $\lt{a}{s}(Z)$ refers to the local time of a semimartingale $Z$. The $\mathrm{d}b(s)$ integrals cancel with each other. Finally, note that the measure $\mathrm{d}_s \lt{a}{s}(Y)$ gives full measure to the set $\{s \, | \, Y_s = a \}$, which is the same as $\{s \, | \, X_s = a + b(s) \}$. This gives (\ref{ltcsecond}) and completes the proof.
\end{proof}

In Ghomrasni and Peskir's work \cite{ghomrasni03}, it is observed that the a similar formula can be derived by formal manipulations in the case when $F$ is instead non-smooth over a curve $c: \mathbb{R} \to \mathbb{R}$  of space, taking values in the time parameter. Indeed, they also mention that a `similar candidate formula' can be formally derived for a general curve $\gamma : [0,1] \to [0,T] \times \mathbb{R}$ over which $F$ is non-smooth. It may be possible to derive a partial result in this direction from Definition \ref{densityexp3}. However, taking $\gamma(t)= (t,t)$, it is easily seen that the function
\begin{equation}
G(s,x) = \left \{ \begin{array}{ll} 1 & \text{if } x>s \\ 0 & \text{otherwise} \end{array} \right.
\end{equation}
is not of bounded variation in the sense we described in Definition \ref{densityexp3}, so the present method does not apply. The author hopes to establish a more general time-space local time on curves formula in a future work.

The following theorem is a result from Ghomrasni \cite{ghomrasni05}, \cite{ghomrasni10}, which we present here using the same idea of proof, but with an alternative construction of the local time-space integral. Other results in Ghomrasni \cite{ghomrasni10} follow directly using the same methods.

\begin{theorem} \label{ghomrasnithm}
Let $H: [0,T] \times \mathbb{R} \to \mathbb{R}$ be local time-space integrable. Then
\begin{equation}
\lim_{\epsilon \to 0} \frac{1}{\epsilon} \int_0^T H(s,X_s) - H(s,X_s - \epsilon) \, \mathrm{d} \left<X,X \right>_s = \int_\mathbb{R} \, \spint{0}{T}{H(s,a)}{\lt{a}{s}},
\end{equation}
where the limit is taken in probability.
\end{theorem}
\begin{proof}
Define $H_\epsilon (s,a) = \frac{1}{\epsilon} \spint{a-e}{a}{H(s,x)}{x}$. Note that $H_\epsilon \to H$ pointwise as $\epsilon \to 0$. Note also that $\frac{\partial}{\partial a}H_\epsilon (s,a) = \frac{1}{\epsilon} \left( H(s,a) - H(s,a-\epsilon) \right)$. Thus by the occupation time formula,
\begin{equation}
\begin{split}
\Lambda(H_\epsilon) &= - \int_\mathbb{R} \spint{0}{T}{\frac{H(s,a) - H(s,a-\epsilon)}{\epsilon}}{_s \lt{a}{s}} \, \mathrm{d} a \\&= \frac{1}{\epsilon} \int_0^T H(s,X_s) - H(s,X_s - \epsilon) \, \mathrm{d} \left<X,X \right>_s.
\end{split}
\end{equation}
Assume that $H$ satisfies the conditions of Definition \ref{densityexp2}. Then we may write
\begin{equation}
\int_\mathbb{R} \spint{0}{T}{\frac{H(s,a) - H(s,a-\epsilon)}{\epsilon}}{_s \lt{a}{s}} \, \mathrm{d} a = \int_\mathbb{R} \spint{0}{T}{ \spint{\mathbb{R}}{}{\frac{h(s,x)}{\epsilon} \indic{a-\epsilon \leq x < a}}{\nu(x)}}{_s \lt{a}{s}} \, \mathrm{d} a.
\end{equation}
Rewriting the indicator function and exchanging the order of integration, we obtain
\begin{equation}
\spint{\mathbb{R}}{}{ \int_\mathbb{R} \spint{0}{T}{ \frac{h(s,x)}{\epsilon} \indic{x < a \leq x + \epsilon}}{_s \lt{a}{s}} \, \mathrm{d} a} {\nu(x)} = \spint{\mathbb{R}}{}{\frac{1}{\epsilon} \int_{x}^{x+\epsilon} \spint{0}{T}{h(s,x)}{_s \lt{a}{s}} \, \mathrm{d} a} {\nu(x)}.
\end{equation}
Left continuity of $h$ in the time variable and right-continuity of local time in the space variable allows us to determine that the function $\phi(a) = \spint{0}{T}{h(s,x)}{_s \lt{a}{s}}$ is right continuous at $x$. Taking a limit as $\epsilon \to 0$, we obtain the result.

The corresponding result when $H$ satisfies instead Definitions \ref{densityexp} or \ref{densityexp3} follow by integration by parts in $s$ then a similar method.
\end{proof}

We now recast a result of Protter and San Martin \cite{protter93} in terms of local time-space integration. We require the local time for a general c\`adl\`ag semimartingale, which can be defined by (\ref{localtimelimit}). In fact, if this semimartingale has jumps of bounded variation, the regularity properties and representation via the Tanaka formula still hold. We refer to \cite{protter04} for more details.

\begin{theorem} \label{prottersanmartin}
Fix a semimartingale $X$. Let $H$ be a continuous process, and $\theta$ a continuous process of bounded variation. Assume that
\begin{equation} \label{newctsass}
\spint{0}{t}{\indic{X_s = \theta_s} }{A_s} = \spint{0}{t}{\indic{X_s = \theta_s} }{\theta(s)} = 0
\end{equation}
for each $t \in \mathbb{R}_+$ almost surely. Then for a refining sequence of partitions $(\pi_n)_{n \geq 1}$ of $[0,t]$, whose mesh converges to zero, we have
\begin{equation}
\lim_{n \to \infty} \sum_{t_i \in \pi_n} H_{t_i} \left( \lt{\theta_{t_{i}}}{t_{i+1}}(X) - \lt{\theta_{t_{i}}}{t_{i}}(X) \right) = \spint{0}{t}{H_s}{_s \lt{\theta}{s}(X)},
\end{equation}
where convergence is uniformly on compacts in probability.
\end{theorem}
\begin{proof}
For each $n$, define pathwise the processes $H^n(s,a) = \sum_{t_i \in \pi_n} H_{t_i} \indic{t_i < s \leq t_{i+1}} \indic{a > 0}$ and $X^n_s = X_s - \sum_{t_i \in \pi_n} \theta(t_i) \indic{t_i \leq s < t_{i+1}}$. It is easily confirmed that $X^n$ are locally uniformly in $\mathcal{H}^1$. Then we may express the sum on the left hand side as 
\begin{equation}
\int_\mathbb{R} \spint{0}{t}{H^n(s,a)}{_s \lt{a}{s}(X^n)} \, \mathrm{d}\delta_0(a).
\end{equation}
By the triangle inequality,
\begin{equation}
\begin{split}
& \left| \spint{0}{t}{H^n(s,a)}{_s \lt{a}{s}(X^n)} - \spint{0}{t}{H_s}{_s \lt{\theta}{s}}(X-\theta) \right| \\
&\leq \left| \spint{0}{t}{\big( H^n(s,a) - H(s,a) \big)}{_s \lt{a}{s}(X^n)} \right| + \left| \spint{0}{t}{H(s,a)}{_s \big(\lt{a}{s}(X^n) - \lt{a}{s}(X - \theta) \big)} \right|.
\end{split}
\end{equation}
The former integral is bounded by
\begin{equation}
\sup_{0 \leq s \leq t} |H^n(s,a) - H(s,a) | \, \text{TV} \left( \mathrm{d}_s  \lt{a}{s}(X^n) \right) \leq \sup_{0 \leq s \leq t} |H^n(s,a) - H(s,a) |  \, \lt{a}{t}(X^n),
\end{equation}
which converges to zero by uniform convergence of $H^n$ to $H$, and boundedness of the sequence of local times.

We now verify convergence in probability of local times to apply Theorem \ref{ctsinsemimg}. Note that the bounded variation component of $X^n$ is $A - \sum_{t_i \in \pi_n} \theta(t_i) \indic{t_i \leq s < t_{i+1}}$, and the bounded variation component of $X - \theta$ is $A - \theta$. Using Lemma \ref{totalvarconv}, we see that 
\begin{equation}
\left| \mathrm{d}_s \left( A_s - \sum_{t_i \in \pi_n} \theta(t_i) \indic{t_i \leq s < t_{i+1}}\right) \right| \to \left| \mathrm{d}_s \left( A_s - \theta_s \right) \right|
\end{equation} 
weakly in probability. To apply Lemma \ref{convergenceoflt}, we must show (\ref{ctsassump}). Given a set $V \in \mathcal{B}(\mathbb{R}_+)$ and a Borel measure $\mu$ on $\mathcal{B}(\mathbb{R}_+)$, the collection $\mathcal{N}_V = \{ U \in \mathcal{B}(\mathbb{R}) \, | \, \mu(U \cap V) = 0 \}$ is a monotone class. We now note that (\ref{newctsass}) implies $(u,v) \in \mathcal{N}_{\{s \, : \, X_s = A_s \}}$ for all $u<v$, where $\mu = \mathrm{d}A_s$ or $\mathrm{d} \theta (s)$. By the monotone class theorem, $\mathcal{N}_V = \mathcal{B}(\mathbb{R})$. 

Finally as $X^n \to X$ uniformly on compacts in probability, we obtain the result.
\end{proof}

\section{Stochastic differential equations involving the local time}

The most fundamental stochastic differential equation involving local time (SDELT) is given simply by
\begin{equation} \label{skbm}
X_t = X_0 + B_t + \beta \lt{a}{t}(X).
\end{equation}
The solution to this equation is the well-known skew Brownian motion with parameter $\beta$ (denoted SkBM), first described in this form by Harrison and Shepp \cite{harrison81}. The SkBM models a particle which travels like a Brownian motion, except at the origin where it hits a permeable barrier with unequal probability of transmission or reflection. If the probability of transmission or reflection is time-dependent, then a natural extension of the SkBM would be a process solving an equation of the form
\begin{equation} \label{sdelt}
X_t = X_0 + \spint{0}{t}{b(s,X_s)}{s} + \spint{0}{t}{\sigma(s,X_s)}{B_s} + \spint{\mathbb{R}}{}{\,\spint{0}{t}{h(s,a)}{_s \lt{a}{s}(X)}}{\nu(a)}.
\end{equation}
In this section, we prove results on existence and pathwise uniqueness of equations of this form. In Section \ref{sdepwsec} we provide some motivation and recap work of other authors on equations of this form. By extending the method of Le Gall, given in the time-independent case, we establish a bijective correspondence between equations of the form (\ref{sdelt}) and standard It\^o equations of the form (\ref{fundamentalsde}) in Lemma \ref{equivlemma}. This requires an extended definition of the local time-space integral, Lemma \ref{densityexp4}, and a change of local time result, Lemma \ref{ltlemma}. Finally we obtain an existence and uniqueness result, Theorem \ref{existenceuniqueness}, in the case when $b = 0$. We then extend this to include a drift term under relatively strong conditions in Theorem \ref{driftexistenceuniqueness}. Finally we discuss the close connections with SDELTs with local time on curves, of the form (\ref{etoremartinezsde}) below.

\subsection{Previous work} \label{sdepwsec}
The study of such equations was initiated by Stroock and Yor \cite{stroock81} in the time homogeneous case, where we have 
\begin{equation} \label{sdelthom}
X_t = X_0 + \spint{0}{t}{\sigma(X_s)}{B_s} + \spint{\mathbb{R}}{}{\lt{a}{t}(X)}{\nu(a)}.
\end{equation}
Their aim was to study the `purity' of certain martingales. Le Gall provided a general treatment of the time homogeneous case in the paper \cite{legall83}, relying upon what has been called the `method of local times' for stochastic differential equations which he introduced in \cite{legall832}. It should be noted that choosing $b=0$, $\sigma = 1$ and $\nu$ a point mass at zero, we obtain (\ref{skbm}) as a special case. See \cite{Lejay06} for a thorough survey of constructions of SkBM.

The results of Le Gall are based upon a bijective transformation which removes the local time component by means of the It\^o-Tanaka formula. Using similar machinery, Rutowski \cite{rutkowski87} and Bass and Chen \cite{bass05} weakened the conditions on the coefficients of (\ref{sdelthom}). We can establish existence and uniqueness results for the transformed equation using the method of local times. The bijection transforms these solutions back to the original equation and so preserves existence and uniqueness.

The general conditions for existence and pathwise uniqueness of the solutions to stochastic differential equations of the form
\begin{equation} \label{fundamentalsde}
\mathrm{d}X_t = b(t,X_t) \mathrm{d}t + \sigma(t,X_t) \mathrm{d}B_t,
\end{equation}
are the Lipschitz and linear growth conditions on both the drift and diffusion coefficients. However in dimension one, in the time-homogeneous case, there are two well-known generalisations attributed to Yamada-Watanabe and Nakao. The Yamada-Watanabe conditions are effectively H\"older conditions, and are known to be sharp. The Nakao condition is more suited to our requirements, as it allows a discontinuous diffusion coefficient, and is effectively a bound on its quadratic variation. We will use a modification of Le Gall's \cite{legall832} proof of this result.

The first development in the time inhomogeneous case is due to Weinryb \cite{weinryb83}, generalised by more recent work by Ouknine and Bouhadou \cite{ouknine13} and \'Etor\'e and Martinez \cite{etore12}, who examined the equation
\begin{equation}
X_t = X_0 + \spint{0}{t}{b(X_s)}{s} + \spint{0}{t}{\sigma(X_s)}{B_s} + \spint{0}{t}{\beta(s)}{_s \lt{0}{s}(X)}.
\end{equation}
More recently,  \'Etor\'e and Martinez \cite{etore18} have extended this to the case when the final term is the local time on a curve, namely by the equation 
\begin{equation} \label{etoremartinezsde}
X_t = X_0 + \spint{0}{t}{b(s,X_s)}{s} + \spint{0}{t}{\sigma(s,X_s)}{B_s} + \spint{0}{t}{\beta(s)}{_s \lt{r}{s}(X)},
\end{equation}
for some curve $r$ of class $C^1$. This uses a slight generalisation of the `local time on curves' formula of Theorem originally due to Peskir \cite{peskir05}, instead of the It\^o-Tanaka formula.

The treatment which we present here uses the change of variables formula Theorem \ref{seconditoform} in place of the It\^o-Tanaka formula, which is crucial in the time homogeneous case. We establish weak existence and pathwise uniqueness results for equations of the form (\ref{sdelt}), which provides existence of a unique strong solution.

\subsection{Existence and uniqueness results} \label{sdeeusection}
Throughout, we let $\sigma: \mathbb{R}_+ \times \mathbb{R} \to \mathbb{R}$ and $b: \mathbb{R}_+ \times \mathbb{R} \to \mathbb{R}$ be bounded and measurable. Further, we assume that $\sigma \geq \epsilon > 0$ for some $\epsilon > 0$. This is sufficient \cite[Co. IX 1.14]{revuz04} to ensure existence and uniqueness in law for the equation (\ref{fundamentalsde}).

Allowing $\sigma$ to approach zero is outside the scope of our method, but it should be noted that if one also appropriately constrains the `soujourn time' at points where $\sigma$ vanishes, then existence, uniqueness and the Markov property can be obtained for SDEs of the form (\ref{fundamentalsde}), see \cite{engelbert84}.

We further assume that $\nu$ is a finite a regular finite Borel measure, $h: \mathbb{R}_+ \times \mathbb{R} \to \mathbb{R}$ is bounded and measurable, and that $\left| h(s,a) \nu ( \left\{ a \right\} )\right| < 1/2$ for all $(s,a) \in \mathbb{R}_+ \times \mathbb{R}$. Essentially this condition cannot be relaxed, as if $\left|h(s,a) \nu ( \left\{ a \right\} )\right| > 1/2$ then even a weak solution cannot exist in general, and if $\left| h(s,a)\nu ( \left\{ a \right\} )\right| = 1/2$ then the equation may admit a weak solution but no strong solution. This was first discussed by Harrison and Shepp \cite{harrison81}, and the paper of Engelbert and Blei \cite{engelbert14} provides a thorough account of these cases in the time homogeneous setting.

We assume that for some strictly increasing function $\rho$, the diffusion coefficient $\sigma$ obeys
\begin{equation} \label{sigmanakao}
(\sigma(t,y) - \sigma(t,x))^2 \leq \left| \rho(y) - \rho(x) \right|,
\end{equation}
for all $(t,x,y) \in \mathbb{R}_+ \times \mathbb{R}^2$. 

Given $h$ and $\nu$ obeying the previous assumptions, we now define the function $F$ which will transform away the local time term of (\ref{sdelt}). This is a natural modification of the method of Le Gall \cite{legall83} to the time dependent case, using exactly the same construction with time-dependent coefficients. His proof uses the It\^o-Tanaka formula, whereas ours uses the corresponding Theorem \ref{seconditoform}. Let
\begin{equation} \label{spacetransform}
F(t,x) = \spint{0}{x}{\exp \left( \zeta(t,a) \right) \psi(t,a)}{a},
\end{equation}
where the functions $\zeta$ and $\psi$ are defined by
\begin{equation}
\zeta(t,x) = \spint{-\infty}{x}{h(t,z)}{\nu^c(z)},
\end{equation}
where $\nu^c$ is the continuous part of $\nu$, and
\begin{equation}
\psi(t,x) = \prod_{z < x} \Big( 1 - h(t,z)\nu(\{z\} )\Big),
\end{equation}
where an empty product is taken to be $1$. We let $F_x$ denote the left derivative of $F$, and define the function $G$ by 
\begin{equation}
G(t,y) = \left[ F(t,\cdot) \right]^{-1} (y).
\end{equation}
To compress notation, we define $\tilde{G}(t,y) = (t,G(t,y))$ and $\tilde{F}(t,x) = (t,F(t,x))$. Then $F \circ \tilde{G}$ and $G \circ \tilde{F}$ are identity maps. We stipulate now that $h$ is differentiable in its former argument when the latter is fixed, with $h_t$ being continuous in time and admitting left limits in space. This ensures that $F_t$ exists and is continuous in time with left limits in space. We can also establish that $G$ can be represented as
\begin{equation} \label{gexpression}
G(t,y) = \spint{0}{y}{\frac{1}{\exp \left( \zeta \circ \tilde{G}(t,a) \right) \psi \circ \tilde{G}(t,a)}}{a}.
\end{equation}
\begin{lemma} \label{equivlemma}
The process $X$ solves the stochastic differential equation with local time 
\begin{equation}
X_t = X_0 + \spint{0}{t}{\sigma(s,X_s)}{B_s} + \spint{\mathbb{R}}{}{\,\spint{0}{t}{h(s,a)}{_s \lt{a}{s}(X)}}{\nu(a)},
\end{equation} 
if and only if the process $Y_t = F(t,X_t)$ solves the stochastic differential equation 
\begin{equation} \label{sdetransformed}
Y_t = Y_0 + \spint{0}{t}{F_{t-} \circ \tilde{G}(s,Y_s)}{s} + \spint{0}{t}{(F_{x} \,\sigma) \circ \tilde{G}(s,Y_s)}{B_s},
\end{equation} 
where $F_{t-}$ represents the left limit of $F_t$ in the space variable.
\end{lemma} 
\begin{proof}
Proving the forward direction consists of applying Theorem \ref{seconditoform}, using representation (\ref{iterated2}). The local time component of (\ref{sdelt}) cancels with the local time-space integral of $F$ by construction.

We show the reverse direction. Note that $X_t = G(t,Y_t)$ by definition. Using the Lebesgue-Stieltjes chain rule then pushing forward,
\begin{equation}
\begin{split}
G_y(t,y) &= G_y(t,x) - \spint{x}{y-}{\frac{1}{(F_x \circ \tilde{G}(t,a))^2} }{_a ( F_x \circ \tilde{G}(t,\cdot))	 } \\ & \hspace{20pt} + \sum_{x < a < y}\frac{1}{(F_{x+} \circ \tilde{G}(t,a))} - \frac{1}{(F_x \circ \tilde{G}(t,a))} \\ & \hspace{60pt}- \frac{1}{(F_x \circ \tilde{G}(t,a))^2} \left( F_{x+} \circ \tilde{G}(t,a) - F_x \circ \tilde{G}(t,a) \right)
 \\[0.8em]
&= \spint{G(t,x)}{G(t,y)-}{\frac{h(t,a)}{F_{x+} (t,a)} }{ \nu(a)},
\end{split}
\end{equation}
where $F_{x+}$ refers to the right-hand space derivative of $F$. Applying Theorem \ref{seconditoform} to the function $G$ and semimartingale $Y$, using the representation (\ref{iterated4}) below, we see
\begin{equation}
\begin{split}
G(t,Y_t) = & \; G(0,Y_0) + \spint{0}{t}{G_{t-} + G_y \cdot \big( F_{t -} \circ \tilde{G} \big) (s,Y_s)}{s} \\&+ \spint{0}{t}{G_y \cdot (F_{x } \,\sigma \circ \tilde{G}) (s,Y_s)}{B_s} + \spint{\mathbb{R}}{}{\spint{0}{t}{\frac{h(s,a)}{F_{x+}(s,a)}}{_s \lt{F(s,a)}{s}(Y)}}{\nu(a)}.
\end{split}
\end{equation}
It then remains to show that
\begin{gather}
G_{t-} +  G_y \cdot \big( F_{t -} \circ \tilde{G} \big)  = 0 ,\\
G_y \cdot (F_{x } \,\sigma \circ \tilde{G}) = \sigma \circ \tilde{G} ,\\
\spint{\mathbb{R}}{}{\spint{0}{t}{\frac{h(s,a)}{F_{x+}(s,a)}}{_s \lt{F(s,a)}{s}(Y)}}{\nu(a)} = \spint{\mathbb{R}}{}{\spint{0}{t}{h(s,a)}{_s \lt{a}{s}(X)}}{\nu(a)}.
\end{gather}
The first two equations follow by passing derivatives under the integral in (\ref{gexpression}), and the one-sided chain rule. To prove the final equation, we use Lemma \ref{ltlemma}.
\end{proof}

It is apparent from the previous proof that we require an extension to the notion of local time-space integral to allow limits which may vary in time. The following extension suffices for our purposes, and demonstrates how further specific cases could be obtained.
\begin{lemma} \label{densityexp4}
Let $H \! : [0,T] \times \rtor$ be left continuous in each argument when the other is fixed. Further, assume that $H$ can be written as
\begin{equation} \label{density4}
H(t,y) - H(t,x) = \spint{\Phi(t,x)}{\Phi(t,y)-}{g(t,a)}{\nu(a)},
\end{equation}
for each $t \in [0,T]$ and $x,y \in \mathbb{R}$. Also let $g$ be left continuous in $t$ for each fixed $a$. Assume that $\Phi$ is invertible as a function of its latter argument for each fixed $t$, and write $\Psi(t,z) = \left[ \Phi(t,\cdot) \right]^{-1} (z)$. Further assume that $\Psi$ is a continuous function of bounded variation in $t$ for each fixed $z$. Then the local time space integral $ \int_{0}^{T} \spint{\mathbb{R}}{}{H(s,x)}{\lt{x}{s}}$ of $H$ is defined by
\begin{equation} \label{iterated4}  - \int_\mathbb{R} \spint{0}{T}{g(u,a)}{_u \lt{\Psi(\cdot,a)}{u}} \, \mathrm{d} \nu(a). \end{equation}
This expression extends the previous representations and the corresponding version of Theorem \ref{seconditoform} holds.
\end{lemma}
\begin{proof}
We follow the proof of Lemma \ref{equiv2}. By pushing forward, we may write (\ref{density4}) as
\begin{equation}
H(t,y) - H(t,x) = \spint{x}{y-}{g(t,\Phi(t,a))}{ (\Psi(t,\cdot)_{\#}\nu)(a)}.
\end{equation}
Now we may replace the right hand side of equation (\ref{firstpartfirstproof}) by
\begin{equation}
\begin{split}
\spint{x}{y}{\spint{\mathbb{R}}{}{\spint{\mathbb{R}}{}{mn \, g(r,\Phi(t,q)) \, \rho(m(t-r)) \, \rho(n(p-q)) }{ (\Psi(t,\cdot)_{\#}\nu)(q) }}{r}}{p}.
\end{split}
\end{equation}
By pushing forward by $\Phi$, we see this becomes
\begin{equation}
\begin{split}
\spint{x}{y}{\spint{\mathbb{R}}{}{\spint{\mathbb{R}}{}{mn \, g(r,q) \, \rho(m(t-r)) \, \rho(n(p-\Psi(t,q))) }{ \nu(q) }}{r}}{p}.
\end{split}
\end{equation}
The proof continues until we take the limit as $n \to \infty$, where we have
\begin{equation} \label{convolved4}
\begin{split}
\Lambda(H^{n}) = - \spint{\mathbb{R}}{}{\spint{\mathbb{R}}{}{\spint{0}{T}{n \, g(s,q) \, \rho(n(p-\Psi(s,q))) }{_s \lt{p}{s}(X) }}{\nu(q)}}{p}.
\end{split}
\end{equation}
Now it remains to show that this is equal to
\begin{equation}
\spint{\mathbb{R}}{}{\spint{\mathbb{R}}{}{\spint{0}{T}{n \, g(s,q) \, \rho(np) }{_s \lt{p}{s}(X-\Psi(s,q)) }}{\nu(q)}}{p}.
\end{equation}
By expanding the local time, we see that (\ref{convolved4}) is equal to
\begin{equation}
\lim_{\epsilon \to 0} \frac{1}{\epsilon}\spint{\mathbb{R}}{}{\spint{\mathbb{R}}{}{\spint{0}{T}{n \, g(s,q) \, \rho(n(p-\Psi(s,q))) \indic{p \leq X_u \leq p + \epsilon} }{\!\left< X,X \right>_s }}{\nu(q)}}{p}.
\end{equation}
Now exchanging the order of integration and translating, then taking the limit back inside, after noting $\left< X,X \right>_s = \left< X - \Psi(\cdot,q),X - \Psi(\cdot,q) \right>_s$ for each fixed $q$, we obtain the right hand side. Finally we allow $n \to \infty$ which finishes the proof.
\end{proof}
The above representation is also useful in establishing the following `change of local time' result. This is a natural extension of the time homogeneous case (see \cite[Ex. VI 1.23]{revuz04})
\begin{lemma} \label{ltlemma}
Let $G: \mathbb{R}_+ \times \mathbb{R} \to \mathbb{R}$ be continuous, and strictly increasing in its latter argument when the former is fixed. Define
\begin{equation}
F(t,x) = \left[ G(t,\cdot) \right]^{-1} (x).
\end{equation} 
Assume $G$ obeys the conditions of Theorem \ref{seconditoform}, with the partial derivative $G_y$ satisfying Definition \ref{densityexp2}. Assume that $t \mapsto F(t,a)$ is of bounded variation for each fixed $a$. For a semimartingale $Y$, define $X_t = G(t,Y_t)$. Then we have
\begin{equation} \label{localtimeeqn}
\spint{0}{t}{G_y(s,F(s,a)+)}{_s\lt{F(\cdot,a)}{s}(Y)} = \lt{a}{s}(X),
\end{equation}
for each $a \in \mathbb{R}$.
\end{lemma}
\begin{proof}
The proof follows by considering, for each fixed $a$, the maps
\begin{equation}
x \mapsto \left| x - a  \right|,
\end{equation}
\begin{equation} \label{absgbar}
H(t,z) = \left| G(t,z + F(t,a) ) - a \right| = \left[ G(t,z + F(t,a) ) - a \right]  \text{sgn}\left(z \right),
\end{equation}
applied to the semimartingales $X$ and $(t,Y - F(\cdot,a))$ respectively. We may then expand using the Tanaka formula and Theorem \ref{seconditoform}, where the local time-space integral of $G$ takes the form of Definition \ref{densityexp2}.

We expand the function $\bar{G}(t,z) = G(t,z + F(t,a))$. By mollifying, taking small increments, then using the Lebesgue-Stieltjes chain rule and taking limits, one can establish that
\begin{equation}
\bar{G}(t,z) - \bar{G}(s,z) = \spint{s}{t}{G_t(u,z + F(u,a))}{u} + \spint{s}{t}{G_y(u,z + F(u,a))}{_u F(u,a)}.
\end{equation}
We can establish likewise, with some technical manipulations, that
\begin{equation}
\begin{split}
H(t,z) - H(s,z) = &\spint{s}{t}{\text{sgn}(z) G_t(u,z + F(u,a))}{u} \\&+ \spint{s}{t}{\text{sgn}(z)G_y(u,z + F(u,a))}{_u F(u,a)}.
\end{split}
\end{equation}
Further, note that
\begin{equation}
\begin{split}
H_z(t,y) - H_z(t,x) & = \spint{x}{y}{\text{sgn}(z) G_{yy}(t,z + F(t,a)) \indic{z \neq 0}}{\nu(z)} \\
& \; \;\; + \left[ G_y(t,F(t,a)+) + G_y(t,F(t,a)-) \right] \delta_0(z) \\
& = \spint{x}{y}{\text{sgn}(z) G_{yy}(t,z + F(t,a))}{\nu(z)} \\
& \; \;\; + 2 \left[ G_y(t,F(t,a)+) \right] \delta_0(z).
\end{split}
\end{equation}
After using the Tanaka formula and expansion of $X_t = G(t,Y_t)$, the result follows.
\end{proof}
\begin{remark}
The above Lemma \ref{ltlemma} holds also when $G_y$ satisfies (\ref{density4}), with obvious modifications to the proof.
\end{remark}
Now that we have established the bijective correspondence of Lemma \ref{equivlemma}, we can prove the main existence and uniqueness result.
\begin{theorem} \label{existenceuniqueness}
Given the previous assumptions on $\sigma,h$ and $\nu$,  we have the existence of a unique strong solution for the equation (\ref{sdelt}).
\end{theorem}
\begin{proof}
By the one-to-one correspondence established in Lemma \ref{equivlemma}, this reduces to considering existence and uniqueness for the equation (\ref{sdetransformed}). We prove the existence of a weak solution and then pathwise uniqueness of solutions. Using the well-known argument of Yamada and Watanabe, these two things together imply the existence of a unique strong solution.

As there is some $\epsilon > 0$ such that $\sigma \geq \epsilon$, and $F_x$ is bounded below by a positive constant due to finiteness of $\nu$, we see that $F_{x } \sigma > \delta$ for some $\delta > 0$. Thus weak existence and uniqueness holds for (\ref{sdelt}).

By (\ref{sigmanakao}) and direct examination of the expression for $F_x$, we see there is some $\tilde{\rho}:\mathbb{R} \to \mathbb{R}$ such that
\begin{equation}
\left| (F_x \sigma) (t,y) -  (F_x \sigma) (t,x) \right|^2 \leq \left| \tilde{\rho} (y) - \tilde{\rho} (x) \right|
\end{equation}
for all $(x,y) \in  \mathbb{R}^2$. Then it follows that
\begin{equation}
\left| (F_x \sigma) \circ \tilde{G} (t,y) -  (F_x \sigma) \circ \tilde{G} (t,x) \right|^2 \leq \left| \tilde{\rho} \circ G(t,y) - \tilde{\rho} \circ G(t,x) \right|,
\end{equation}
for all $(t,x,y) \in \mathbb{R}_+ \times \mathbb{R}^2$. By the theory of \cite{legall832}, it suffices to show that
\begin{equation}
\mathbb{E} \left[ \spint{0}{t}{\indic{Z_s > 0} Z_s^{-1} }{\! \left< Z,Z \right>_s} \right] < \infty,
\end{equation}
where $Z = \bar{Y} - Y$ for any two solutions $\bar{Y},Y$ of (\ref{sdetransformed}). We see 
\begin{equation}
\begin{split}
\mathbb{E} \left[ \spint{0}{t}{\indic{Z_s > \delta} Z_s^{-1} }{\! \left< Z,Z \right>_s} \right]& \leq \\& \hspace{-20pt} \mathbb{E} \left[ \spint{0}{t}{\indic{Z_s > \delta} Z_s^{-1} \left(\tilde{\rho} \circ G (s,\bar{Y}_s) - \tilde{\rho} \circ G (s,Y_s) \right) }{s} \right].
\end{split}
\end{equation}
Define processes $\bar{X}$ and $X$ by $\bar{X}_s = G(s,\bar{Y}_s)$ and $X_s = G(s,Y_s)$. Take $Z^u = X + u (\bar{X} - X)$ and note that $Z^u$ is a semimartingale. Take a bounded sequence of increasing $C^\infty$ functions $f_n$ which converge pointwise to $f$ except possibly on the discontinuity set of $F$. Then note that
\begin{equation} \label{rabbit1}
\begin{split}
&\mathbb{E} \left[ \spint{0}{t}{\indic{Z_s > \delta} Z_s^{-1} \left( \tilde{\rho}(\bar{X}_s) - \tilde{\rho}(X_s) \right) }{s} \right] \leq K \mathbb{E} \left[ \spint{0}{t}{ \spint{0}{1}{\frac{\partial f_n}{\partial x} (Z^u_s)  }{u} }{s} \right],
\end{split}
\end{equation}
where $K$ is the Lipschitz constant of $G$ in the space variable. We employ the occupation time formula with respect to $Z^u$ to see
\begin{equation}
\mathbb{E} \left[ \spint{0}{1}{ \spint{0}{t}{\frac{\partial f_n}{\partial x} (Z^u_s)  }{s} }{u} \right] \leq \frac{1}{\epsilon^2} \spint{0}{1}{ \spint{\mathbb{R}}{}{ \frac{\partial f_n}{\partial x} (a)  \mathbb{E} \left[ \lt{a}{t}(Z^u) \right] }{a} }{u} .
\end{equation}
We note that, using Lemma \ref{ltlemma}, $\mathbb{E} \left[ \lt{a}{t}(Z^u) \right]$ is uniformly bounded in $(a,u)$ by some constant $C_t < \infty$, and thus the expression (\ref{rabbit1}) is bounded by
\begin{equation}
\frac{2KC}{\epsilon^2} \max_{a} \left| f(a) \right|.
\end{equation}
The conclusion follows after letting $n \to \infty$ and $\delta \to 0$, noting that the discontinuity set of $f$ is at most countable.
\end{proof}

The analogue of (\ref{sdelt}) obtained by replacing the left local time with symmetric local time appears in the literature, and may also produce a more realistic model. We employ the following trick, originally demonstrated in \cite{bass05}, to transfer our results to the symmetric case without repeating the original analysis.
\begin{lemma}
Assume that $X$ solves the SDE (\ref{sdelt}). Then the symmetric local time $\tilde{\lt{}{}}$ is related to the right local time by the relation
\begin{equation}
\tilde{\lt{a}{t}} = \spint{0}{t}{\Big[ 1 - h(s,a) \nu(\left\{ a \right\}) \Big]}{_s \lt{a}{s}(X)}.
\end{equation}
\end{lemma}
\begin{proof}
By definition, the symmetric local time is
\begin{equation} \label{symdef}
\tilde{\lt{a}{t}} = \frac{1}{2} \left( \lt{a}{t} + \lt{a-}{t} \right).
\end{equation}
For any semimartingale $Z$, with canonical decomposition $V + M$ into an adapted process of bounded variation and a local martingale respectively, we see from \cite[Thm. VI 1.7]{revuz04} that
\begin{equation}
\lt{a}{t} - \lt{a-}{t} = 2 \spint{0}{t}{\indic{Z_s = a}}{V_s}.
\end{equation}
Using the SDE (\ref{sdelt}), we see this is
\begin{equation}
\begin{split}
\lt{a}{t} - \lt{a-}{t} &= 2\spint{\mathbb{R}}{}{\spint{0}{t}{h(s,z)\indic{X_s = a}}{_s \lt{z}{s}(X)}}{\nu(z)} \\ &= 2\spint{\{a\}}{}{\spint{0}{t}{h(s,z)}{_s \lt{z}{s}(X)}}{ \nu(z)} = 2\spint{0}{t}{h(s,a) \nu(a)}{_s \lt{a}{s}(X)}.
\end{split}
\end{equation}
The conclusion then follows by substituting this expression into (\ref{symdef}).
\end{proof}

We conclude by including a classical drift term. Assume now that $\sigma, \nu$ and $h$ obey all the previous assumptions. We deal with the SDE (\ref{sdelt}). Assume that $b$ and $\sigma$ satisfy the following:
\begin{enumerate}
\item $\frac{b}{\sigma^2}(t,\cdot) \in L^1$ for each $t$;
\item $t \mapsto b(t,a)$ is $C^1$ for each fixed $a$, with $b_t(t,a)$ left continuous and admitting left limits in space;
\item $t \mapsto \sigma(t,a)$ is $C^1$ for each fixed $a$, with $\sigma_t(t,a)$ left continuous and admitting left limits in space;
\end{enumerate}

\begin{theorem} \label{driftexistenceuniqueness}
Given the previous assumptions on $\sigma,b,h$ and $\nu$, we have existence and uniqueness for the equation (\ref{sdelt}).
\end{theorem}
\begin{proof}
As $ \sigma \geq \epsilon$, we note that 
\begin{equation}
\spint{0}{t}{b(s,X_s) }{s} = \spint{0}{t}{\frac{b}{\sigma^2}(s,X_s) \cdot \sigma^2(s,X_s)}{s} = \spint{\mathbb{R}}{}{\spint{0}{t}{\frac{b}{\sigma^2}(s,a)}{_s \lt{a}{s}}}{a},
\end{equation}
by the occupation time formula. We replace $h$ by $h + b/\sigma^2$, except possibly on some set of Lebesgue measure zero in the $a$-variable, and then apply Theorem \ref{existenceuniqueness}.
\end{proof}


\bibliographystyle{../includes/napalike}
\bibliography{../includes/bib}


\par \leftskip=24pt

\vspace{12pt}

\ni Daniel Wilson \\
School of Mathematics \\
The University of Manchester \\
Oxford Road \\
Manchester M13 9PL \\
United Kingdom \\
\texttt{daniel.wilson-2@manchester.ac.uk}

\par

\end{document}